\documentclass[a4paper]{article}

\usepackage[utf8]{inputenc}
\usepackage{amsmath, amssymb, mathabx, mathtools,amsthm}
\usepackage{hyperref}
\usepackage{graphicx}
\usepackage{xcolor}
\usepackage[ruled]{algorithm2e}
\usepackage[hmarginratio=3:2]{geometry}
\usepackage{enumerate}
\newtheorem{definition}{Definition}[section]
\newtheorem{theorem}[definition]{Theorem}
\newtheorem{corollary}[definition]{Corollary}
\newtheorem{lemma}[definition]{Lemma}
\newtheorem{remark}[definition]{Remark}
\newtheorem{example}[definition]{Example}

\usepackage{chngcntr}
\counterwithin{figure}{section}
\counterwithin{table}{section}

\newcommand{\defas}{:=}
\newcommand{\ind}{\chi}

\newcommand{\nlOp}{\mathcal{L}}

\newcommand{\nlDom}{\Omega}
\newcommand{\nlBound}{\mathcal{I}}

\newcommand{\completeDom}{\nlDom \cup \nlBound}
\newcommand{\varOp}{A}
\newcommand{\varForce}{F}

\newcommand{\xb}{\mathbf{x}}
\newcommand{\yb}{\mathbf{y}}

\newcommand{\Vb}{\textbf{V}}
\newcommand{\Wb}{\textbf{W}}

\newcommand{\Nbb}{\mathbb{N}}

\newcommand{\kernel}{\gamma}
\newcommand{\kernelt}{\gamma^t}
\newcommand{\kernelij}{\gamma_{ij}}

\newcommand{\kerneliI}{\gamma_{iI}}

\newcommand{\advar}{v}
\newcommand{\trialSpace}{V}
\newcommand{\testSpace}{V_c}

\newcommand{\weakSol}{u}

\newcommand{\shape}{\Gamma}
\newcommand{\initialShape}{\Gamma_0}
\newcommand{\shapet}{\shape_t}
\newcommand{\shapetW}{\shapet^{\Wb}}

\newcommand{\shapeSpace}{\mathcal{A}}

\newcommand{\R}{\mathbb{R}}
\newcommand{\Rd}{\mathbb{R}^d}

\newcommand{\Fb}{\mathbf{F}}
\newcommand{\Fbt}{\Fb_{t}}
\newcommand{\FbtW}{\Fb_{t}^{\Wb}}
\newcommand{\FbtV}{\Fb_{t}^{\Vb}}
\newcommand{\FbV}{\Fb^{\Vb}}
\newcommand{\Ftzero}{\mathbf{F}_{0}}

\newcommand{\vecfields}{C_0^k(\completeDom,\R^d)} 
\newcommand{\vecfieldsspecific}{C_0^2(\completeDom,\R^d)}

\newcommand{\data}{\bar{\weakSol}}
\newcommand{\intOp}{\varOp_{\shape}}
\newcommand{\intOptW}{\varOp_{\shapetW}}
\newcommand{\intForce}{\varForce_{\shape}}
\newcommand{\intForcetW}{\varForce_{\shapetW}}
\newcommand{\secondVarForce}{\widetilde{\varForce}}
\newcommand{\secondIntForcetW}{\secondVarForce_{\shapetW}}
\newcommand{\intKernel}{\kernel_{\shape}}

\newcommand{\objFun}{J}
\newcommand{\redFun}{J^{red}}

\newcommand{\lagrangian}{\mathfrak{L}}
\newcommand{\reallagrangian}{\mathfrak{G}}
\newcommand{\Tb}{\mathbf{T}}

\DeclareMathOperator{\Id}{{\bf Id}}
\DeclareMathOperator{\grad}{\nabla}
\DeclareMathOperator{\di}{div}
\DeclareMathOperator{\hess}{Hess}
\DeclareMathOperator{\supp}{supp}
\DeclareMathOperator{\trace}{tr}

\newcommand{\xt}{\xi^t}
\newcommand{\objFunDer}{\mathfrak{\objFun}}
\newcommand{\objFunDertW}{\mathfrak{\objFun}_{\shapetW}}
\newcommand{\varForceDer}{\mathfrak{\varForce}}
\newcommand{\varForceDertW}{\mathfrak{\varForce}_{\shapetW}}
\newcommand{\varOpDer}{\mathfrak{\varOp}}
\newcommand{\varOpDertW}{\mathfrak{\varOp}_{\shapetW}}

\newcommand{\Vbt}{\Vb^t}
\newcommand{\reg}{j_{reg}}
\newcommand{\perimeter}{j_{per}}
\newcommand{\normal}{n}

\newcommand{\holdAll}{D}

\title{\textbf{Second Order Shape Optimization for an Interface Identification Problem constrained by Nonlocal Models}}
\author{Matthias Schuster\thanks{Universitaet Trier, D-54286 Trier, Germany; Email: schusterm@uni-trier.de,  volker.schulz@uni-trier.de}\hspace{2mm}\href{https://orcid.org/0000-0002-9355-1076}{\includegraphics[scale=0.06]{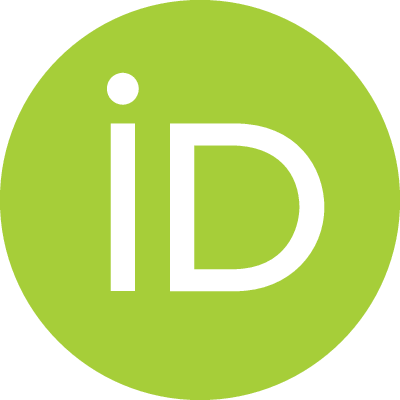}} \and Volker Schulz\footnotemark[1]\hspace{2mm}\href{https://orcid.org/0000-0001-7665-130X}{\includegraphics[scale=0.06]{orcid.eps}} }
\date{}

\begin{document}
	\maketitle
	\small
	\textbf{Abstract.}
	Since shape optimization methods have been proven useful for identifying interfaces in models governed by partial differential equations, we show how shape optimization techniques can also be applied to an interface identification problem constrained by a nonlocal Dirichlet problem. Here, we focus on deriving the second shape derivative of the corresponding reduced functional and we further investigate a second order optimization algorithm.
	~\\
	\textbf{Keywords.} shape optimization, second shape derivative, nonlocal convection-diffusion.
	\normalsize
	
	\section{Introduction}
	Nonlocal operators are typically integral operators and therefore allow interactions between two distinct points in space. Moreover, the solution of a nonlocal problem in general needs to fulfill less regularity requirements compared to the solution of a system of partial differential equations. Thus, some physical phenomena like fracture propagation in continuum mechanics\cite{silling2000,javili2019peridynamics} or anomalous diffusion\cite{suzuki2022,brockmann2008} effects can be modeled more accurately by nonlocal equations compared to their 'classic' description by partial differential equations. Moreover, nonlocal models have been successfully applied to image denoising\cite{nldenoising1,delia_bilevel_image_denoising}, neural networks\cite{NN1,NN2} or stochastic processes\cite{metzler2000random,delia2017nonlocal}, just to name a few.\\
	Shape optimization techniques are mostly developed in the context of PDE-constrained optimization problems and have been investigated in areas like aerodynamics\cite{schmidt_diss, aerodynamics1}, acoustics\cite{acoustics1,acoustics2,acoustics3}, fluid dynamics\cite{fluid_dynamics, Luka_diss} and many more.
	In \cite{shape_paper}, we already formulated the first shape derivative of an interface identification problem that is governed by a nonlocal Dirichlet problem and we additionally investigated a first order shape optimization algorithm. In this work, we continue by deriving a second shape derivative, which is then used in a Newton-like second order shape optimization approach.\\
	The first order shape derivative method of \cite{shape_paper} was also employed for interface identification constrained by an energy-based Local-to-Nonlocal coupling in \cite{LtN_paper}. Moreover, in the context of Peridynamics, a machine learning approach to detect an interface was tested in \cite{nader_ml}. Here, Nader et. al. made use of finite differences to solve a one-dimensional variable horizon Local-to-Nonlocal coupling, and a convolutional neural network was trained to assign domain affiliations to nodes.\\
	In the next chapter, we start by introducing nonlocal Dirichlet problems and shortly discuss the well-posedness of the corresponding weak formulation. In Chapter \ref{chap:int_id} we define an interface identification problem that is constrained by a nonlocal Dirichlet problem. Then, we present necessary basics of shape optimization in Chapter \ref{chap:basics_shape_opt} followed by a description of the averaged adjoint method(AAM) in Section \ref{chap:AAM_second_der}, which we will use to compute the second shape derivative of the reduced functional corresponding to the constrained interface identification problem. After that, we show in  Chapter \ref{chap:AAM_proof}, that, under natural assumptions, the prerequisites of AAM are satisfied. Lastly, we investigate in Section \ref{chap:num_exp} a second order optimization approach, that is partly a consequence of AAM. 
	\section{Nonlocal Terminology and Framework}
	In this work, we consider the \emph{nonlocal convection-diffusion operator}
	\begin{align}\label{def:nlOp}
		-\nlOp \weakSol(\xb) \defas \int_{\R^d} \weakSol(\xb)\kernel(\xb,\yb) - \weakSol(\yb)\kernel(\yb,\xb) ~d\yb d\xb.
	\end{align}
	Here, the nonnegative function $\kernel:\Rd \times \Rd \rightarrow [0,\infty)$ is the so-called \emph{kernel}, which plays a vital role in the nonlocal theory. Moreover, problems involving the nonlocal operator \eqref{def:nlOp} have been extensively studied in the literature (see, e.g., \cite{DuAnalysis,vollmann_diss}). We define a \emph{nonlocal steady-state nonhomogeneous Dirichlet problem} as
	\begin{align}
		\begin{split}\label{def:nlDirProb}
			\textit{Find a function } &\weakSol \textit{ s.t.}\\
			-\nlOp \weakSol &= f \quad  \text{on } \nlDom \\
			\weakSol &= g \quad \text{on } \nlBound,
		\end{split}
	\end{align}
	where we set the so-called \emph{nonlocal boundary} (or \emph{nonlocal interaction domain}) $\nlBound$ as
	\begin{align*}
		\nlBound \defas \{\yb \in \Rd \setminus \nlDom: \int_{\nlDom} \kernel(\xb,\yb) + \kernel(\yb,\xb) ~d\xb > 0 \},
	\end{align*}
which contains all points outside of the domain $\nlDom$, that interact with at least one point inside of $\nlDom$ through the kernel $\kernel$.
In this section, we will recall the basic theory on problems of type \eqref{def:nlDirProb}. We start by deriving a weak formulation of problem \eqref{def:nlDirProb}. Given appropriate spaces $\trialSpace(\completeDom)$ and $\testSpace(\completeDom)$, that are defined later in this section, we multiply the first equation of \eqref{def:nlDirProb}, where we assume $\weakSol \in \trialSpace(\completeDom)$, with a test vector $\advar \in \testSpace(\completeDom)$ and then integrate over $\nlDom$, which yields
	\begin{align*}
		\int\limits_{\nlDom} \int\limits_{\completeDom} \advar(\xb)\left(\weakSol(\xb)\kernel(\xb,\yb) - \weakSol(\yb)\kernel(\yb,\xb) \right) ~d\yb d\xb = \int\limits_{\nlDom} f(\xb) \advar(\xb) ~d\xb. 
	\end{align*}
Here, we can rewrite the left-hand side using $\advar=0$ on $\nlBound$ and applying Fubini's theorem as follows
\begin{align*}
	&\int\limits_{\nlDom} \int\limits_{\completeDom} \advar(\xb)\left(\weakSol(\xb)\kernel(\xb,\yb) - \weakSol(\yb)\kernel(\yb,\xb) \right) ~d\yb d\xb\\
	&= \frac{1}{2} \int\limits_{\completeDom} \int\limits_{\completeDom} \left(\advar(\xb) - \advar(\yb)\right)\left(\weakSol(\xb)\kernel(\xb,\yb) - \weakSol(\yb)\kernel(\yb,\xb) \right) ~d\yb d\xb. 
\end{align*}
This reformulation is numerically more stable for singular symmetric kernels, which will be introduced later in this chapter.
Now, before we present the definition of a weak solution to \eqref{def:nlDirProb}, we define the bilinear operator
\begin{align*}
	\varOp(\weakSol,\advar) \defas &= \frac{1}{2}  \iint\limits_{(\completeDom)^2} \left(\advar(\xb) - \advar(\yb)\right)\left(\weakSol(\xb)\kernel(\xb,\yb) - \weakSol(\yb)\kernel(\yb,\xb) \right) ~d\yb d\xb
\end{align*}
and the linear functional
\begin{align*}
	\varForce(\advar) \defas \int\limits_{\nlDom} f(\xb) \advar(\xb) ~d\xb.
\end{align*}
Then, we set the (semi-)norm $|||\cdot|||$ and the norm $||\cdot||_{\trialSpace(\completeDom)}$ as
\begin{align*}
	|||\weakSol||| \defas \sqrt{\varOp(\weakSol,\weakSol)} \text{ and }
	||\weakSol||_{\trialSpace(\completeDom)} \defas |||\weakSol||| + ||\weakSol||_{L^2(\completeDom)}.
\end{align*}
Additionally, we define the nonlocal energy space $\trialSpace(\completeDom)$ and the nonlocal volume-constrained energy space $\testSpace(\completeDom)$ in the following way: 
\begin{align*}
	\trialSpace(\completeDom) \defas \{\weakSol \in L^2(\completeDom): |||\weakSol||| < \infty \} \text{ and }
 	\testSpace(\completeDom) \defas \{\weakSol \in \trialSpace(\completeDom): \weakSol = 0 \text{ on } \nlBound \}.
 \end{align*}
\begin{definition}
	Given $f \in L^2(\nlDom)$ and $g \in \trialSpace(\completeDom)$, if $\widetilde{\weakSol} \in \testSpace(\completeDom)$ satisfies
		\begin{align*}
			\varOp(\widetilde{\weakSol},\advar) = &\varForce(\advar) - \varOp(g,\advar)  \textit{ for all } \advar \in \testSpace(\completeDom),
	\end{align*}
	then we call $\weakSol \defas \widetilde{\weakSol} + g \in \trialSpace(\completeDom)$ weak solution of the nonlocal Dirichlet problem \eqref{def:nlDirProb}.
\end{definition}~\\
For $\delta >0$ we set $B_{\delta}(\xb) \defas \{\yb \in \Rd: ||\xb-\yb||_2 < \delta\}.$
Then, in this work every kernel $\kernel$ is assumed to fulfill:
\begin{itemize}
	\item [(K1)] There exists a constant $\delta>0$ and a function $\varphi:\Rd \times \Rd \rightarrow [0,\infty)$ such that 
	\begin{align*}
		\kernel(\xb,\yb) = \varphi(\xb,\yb)\ind_{B_{\delta}(\xb)}(\yb).
	\end{align*}
	\item[(K2)] There exist constants $\epsilon \in (0,\delta)$ and $\underline{\kernel} > 0$ with
	\begin{align*}
		\kernel(\xb,\yb) \geq \underline{\kernel} \text{ for } \xb \in \nlDom \text{ and } \yb \in B_{\epsilon}(\xb).
	\end{align*}
\end{itemize}
Further, we consider two types of kernel classes:
\begin{itemize}
	\item \emph{Integrable kernels}:\\
	Set $\kernel^s(\xb,\yb) \defas \frac{1}{2}\left(\kernel(\xb,\yb) + \kernel(\yb,\xb) \right)$ and $\kernel^a(\xb,\yb) \defas \frac{1}{2}\left(\kernel(\xb,\yb) - \kernel(\yb,\xb)\right)$. Then, $\kernel$ is called \emph{integrable}, if the following requirements are fulfilled:
	\begin{enumerate}[(1)]
		\item There exist constants $\underline{\kernel}^s, \overline{\kernel}^s > 0$ such that
		\begin{align*}
			\underline{\kernel}^s = \inf_{\xb \in \nlDom} \int_{\completeDom} \kernel^s(\xb,\yb) ~d\yb < \infty \text{ and } \sup_{\xb \in \nlDom} \int_{\completeDom} (\kernel^s(\xb,\yb))^2 ~d\yb = (\overline{\kernel}^s)^2.
		\end{align*}
		\item There exists constants $\underline{\kernel}^a \in \R$ and $\overline{\kernel}^a > 0$ with
		\begin{align*}
			\underline{\kernel}^a = \inf_{\xb \in \nlDom} \int_{\completeDom} \kernel^a(\xb,\yb) ~d\yb \text{ and } \sup_{\xb \in \nlDom} \int_{\completeDom} |\kernel^a(\xb,\yb)| ~d\yb = \overline{\kernel}^a.
		\end{align*}
		\item The lower boundaries $\underline{\kernel}^s$ and $\underline{\kernel}^a$ satisfy 
		$$\underline{\kernel}^s + \underline{\kernel}^a > 0.$$
	\end{enumerate}
	\item \emph{Singular symmetric kernels}:\\
	There exist positive constants $0 < \kernel_* \leq \kernel^* < \infty$ and $s \in (0,1)$ such that
	\begin{align}
		\label{ineq:sing_sym_kernel}
		\kernel_* \leq \kernel(\xb,\yb)||\xb-\yb||_2^{d+2s} \leq \kernel^* \text{ for all } \xb,\yb \in \completeDom.
	\end{align}
	Furthermore, the kernel $\kernel$ is symmetric.
\end{itemize}
In \cite{DuAnalysis, vollmann_diss} it is shown, that there is a unique weak solution to the Dirichlet problem \eqref{def:nlDirProb}, if the corresponding kernel fulfills (K1) and (K2) and additionally fits into one of those two categories. 
\begin{remark}\label{remark:norm_equiv}
	In the volume-constrained space $\testSpace(\completeDom)$ the norms $||\cdot||_{L^2(\completeDom)}$ and $|||\cdot|||$ are equivalent, if we have chosen an integrable kernel. Thus, it holds that
	\begin{align*}
		\weakSol \in (\testSpace(\completeDom),|||\cdot|||) \Leftrightarrow \weakSol \in (L_c^2(\completeDom),||\cdot||_{L^2(\completeDom)}),
	\end{align*}
where $L^2_c(\completeDom)\defas \{\weakSol \in L^2(\completeDom): \weakSol = 0 \text{ on } \nlBound\}$.\\
On the other hand, if $\kernel$ fulfills the requirements of a singular symmetric kernel, we get norm equivalence between $|||\cdot|||$ and $|\cdot|_{H^s(\completeDom)}$ on $\testSpace(\completeDom)$ and
\begin{align*}
	\weakSol \in (\testSpace(\completeDom),|||\cdot|||) \Leftrightarrow \weakSol \in (H^s_c(\completeDom), |\cdot|_{H^s(\completeDom)}),
\end{align*}
where the norm $|\cdot|_{H^s(\completeDom)}$ as well as the volume-constrained fractional Sobolev space $H^s_c(\completeDom)$ are defined as
\begin{align*}
	|\weakSol|_{H^s(\completeDom)}^2 \defas \iint_{(\completeDom)^2} \frac{(\weakSol(\xb) - \weakSol(\yb))^2}{||\xb - \yb||_2^{d+2s}} ~d\yb d\xb \text{ and }
	H_c^s(\completeDom) \defas \{\weakSol \in L_c^2(\completeDom): |\weakSol|_{H^s(\completeDom)}  < \infty \}.
\end{align*}
For the proof of these assertions, we again refer to \cite{DuAnalysis,vollmann_diss}.
\end{remark}
\section{Interface Identification constrained by Nonlocal Models}
\label{chap:int_id}
PDE-constrained interface identification is a well-researched shape optimization problem (see, e.g., \cite{welker_diss, shapes_geometries}) and in the following we will investigate interface identification governed by nonlocal models.\\
Here as illustrated in Figure \ref{fig:interface}, we assume that $\nlDom$ is decomposed in two open and nonempty sets $\nlDom_1, \nlDom_2 \subset \Rd$, such that $\nlDom = \nlDom_1 \dot{\cup} \shape \dot{\cup} \nlDom_2$ holds, where $\shape \defas \partial \nlDom_1 \cap \partial \nlDom_2$ is called \emph{interface}. 
\begin{figure}[h!]
	\centering
	\def\svgwidth{0.3\textwidth}
	{%% Creator: Inkscape 1.1 (c68e22c387, 2021-05-23), www.inkscape.org
%% Accompanies image file 'interface_problem.pdf' (pdf, eps, ps)
\begingroup%
  \makeatletter%
  \providecommand\color[2][]{%
    \errmessage{(Inkscape) Color is used for the text in Inkscape, but the package 'color.sty' is not loaded}%
    \renewcommand\color[2][]{}%
  }%
  \providecommand\transparent[1]{%
    \errmessage{(Inkscape) Transparency is used (non-zero) for the text in Inkscape, but the package 'transparent.sty' is not loaded}%
    \renewcommand\transparent[1]{}%
  }%
  \providecommand\rotatebox[2]{#2}%
  \newcommand*\fsize{\dimexpr\f@size pt\relax}%
  \newcommand*\lineheight[1]{\fontsize{\fsize}{#1\fsize}\selectfont}%
  \ifx\svgwidth\undefined%
    \setlength{\unitlength}{358.27765652bp}%
    \ifx\svgscale\undefined%
      \relax%
    \else%
      \setlength{\unitlength}{\unitlength * \real{\svgscale}}%
    \fi%
  \else%
    \setlength{\unitlength}{\svgwidth}%
  \fi%
  \global\let\svgwidth\undefined%
  \global\let\svgscale\undefined%
  \makeatother%
  \begin{picture}(1,1.08425571)%
    \lineheight{1}%
    \setlength\tabcolsep{0pt}%
    \put(0,0){\includegraphics[width=\unitlength,page=1]{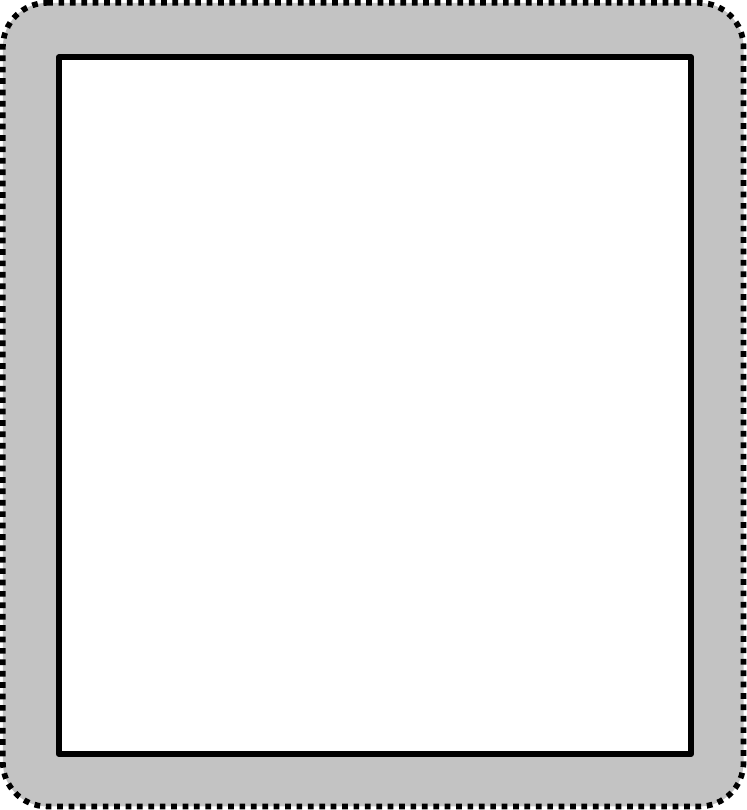}}%
    \put(0.47347597,0.55462095){\makebox(0,0)[lt]{\lineheight{1.25}\smash{\begin{tabular}[t]{l}$\nlDom_1$\end{tabular}}}}%
    \put(0.4719123,1.015){\makebox(0,0)[lt]{\lineheight{1.25}\smash{\begin{tabular}[t]{l}{\small $\nlBound$}\end{tabular}}}}%
    \put(0,0){\includegraphics[width=\unitlength,page=2]{interface_problem.pdf}}%
    \put(0.69035752,0.8854087){\makebox(0,0)[lt]{\lineheight{1.25}\smash{\begin{tabular}[t]{l}$\nlDom_2$\end{tabular}}}}%
    \put(0.62014405,0.28624897){\makebox(0,0)[lt]{\lineheight{1.25}\smash{\begin{tabular}[t]{l}{\color{blue} $\shape$}\end{tabular}}}}%
  \end{picture}%
\endgroup%
}
	\caption{The interface $\shape$, which is colored in blue, decomposes the white area $\nlDom$ into subdomains $\nlDom_1$ and $\nlDom_2$. Further, note that the nonlocal boundary $\nlBound$ is depicted in gray.}
	\label{fig:interface}
\end{figure}
Since shape optimization is typically done w.r.t. two or three dimensional domains, we suppose $d \in \{2,3\}$. 
Now, given $\data \in H^2(\nlDom)$, we define \emph{interface identification governed by a nonlocal Dirichlet problem} as follows:
\begin{align}	
	\min_{\weakSol, \shape} \objFun(\weakSol, \shape) \defas j(\weakSol,\shape) + \nu \perimeter(\shape) &\defas \int_{\nlDom} \left(\weakSol - \data\right)^2 ~d\xb + \nu \int_{\shape} 1 ~ds \nonumber \\
	\begin{split}
		\label{nlDirInt}
		\textit{s.t. } -\nlOp_{\shape} \weakSol &= f_\shape \quad \textit{on } \nlDom,\\
		\weakSol &= 0 \quad \textit{ on } \nlBound,
	\end{split}
\end{align}
where 
\begin{align}
	\begin{split}\label{def:int_dep_nlOp}
	-\nlOp_{\shape} \weakSol(\xb) &\defas \int_{\Rd} \weakSol(\xb) \intKernel(\xb,\yb) - \weakSol(\yb)\intKernel(\yb,\xb) ~d\yb \text{ with}\\
	\intKernel &\defas \sum_{i,j=1,2} \kernelij \ind_{\nlDom_i \times \nlDom_j} + \sum_{i=1,2} \kerneliI \ind_{\left(\nlDom_i \times \nlBound\right) \cup \left(\nlBound \times \nlDom_i\right)}. 
\end{split}
\end{align}
For ease of presentation, we only consider homogeneous Dirichlet problems. However, the nonhomogeneous case can be handled analogously.
As we can see in \eqref{def:int_dep_nlOp}, the operator $-\nlOp_{\shape}$ is dependent on the interface, since the choice of kernel $\intKernel(\xb,\yb)$ varies depending on the location of $\xb$ and $\yb$. 
Furthermore, we supposed that the forcing term can be expressed as $f_\shape = f_1 \ind_{\nlDom_1} + f_2 \ind_{\nlDom_2}$, where $f_1, f_2 \in H^2(\nlDom)$.
Then, we denote the corresponding variational formulation of the Dirichlet problem \eqref{nlDirInt} as
\begin{align}\label{eq:weak_form_int}
	\intOp(\weakSol, \advar) = \intForce(\advar) \textit{ for all } \advar \in \testSpace(\completeDom),
\end{align}
where the subscripts highlight that the kernel parameters as well as the forcing term are dependent on the decomposition imposed by the interface $\shape$. 
The objective functional consists of two terms: The first one, which we refer to as $j$, is a tracking-type functional, since we aim to find the interface $\shape$ such that the solution to \eqref{eq:weak_form_int} approximates the given data $\data$ as well as possible. The second integral of the objective function, which we denote by $\perimeter$, is called \emph{perimeter regularization}, which is often used to avoid ill-posedness of the problem (see \cite{perimeter}). Previously, as mentioned above, first-order shape optimization algorithms for the interface identification problem \eqref{nlDirInt} have been investigated in \cite{shape_paper}.\\ 
Since for every suitable interface $\shape$ the nonlocal Dirichlet problem has a unique solution, the problem \eqref{nlDirInt} is only dependent on the shape $\shape$, i.e., we can rewrite \eqref{nlDirInt} as a reduced problem
\begin{align*}
	\min_{\shape} \redFun(\shape) \defas \objFun(\weakSol(\shape), \shape),
\end{align*} 
where $\weakSol(\shape)$ satisfies \eqref{eq:weak_form_int} given $\shape$. Therefore, we continue by recalling necessary basic tools of shape optimization in order to derive the second shape derivative of the reduced functional $\redFun$ and a second order shape optimization algorithm.
\section{Basics of Shape Optimization}
\label{chap:basics_shape_opt}
To define shape derivatives we make use of a family of mappings $\{\Fbt\}_{t\in[0,T]}$, where $\Fbt: \overline{\nlDom} \rightarrow \Rd$, $\Ftzero(\xb) = \xb$ for $\xb \in \overline{\nlDom}$ and $T > 0$ is supposed to be sufficiently small. In this work, we choose to apply the perturbation of identity:
\begin{definition}[Perturbation of Identity]
	\label{def:pert_id}
	Given a vector field $\Vb \in C_0^k(\nlDom,\Rd)$, $k \in \Nbb$ and $t \geq 0$, the perturbation of identity is defined as
	\begin{align*}
		\Fbt^{\Vb}:\overline{\nlDom} \rightarrow \Rd, ~\Fbt^{\Vb}(\xb) \defas \xb + t \Vb(\xb),
	\end{align*}
where $C_0^k(\nlDom,\Rd) \defas \{\Vb \in C^k(\nlDom,\Rd): \Vb \text{ vanishes at the boundary} \}$. 
\end{definition}~\\
Here, ${\Vb \in C^k(\nlDom,\Rd)}$ \emph{vanishes at the boundary}, if the set $\{\xb \in \nlDom: |\partial^{\alpha}\Vb(\xb)| \geq \epsilon \}$ is compact for all $\epsilon > 0$ and partial derivatives $\partial^\alpha\defas\partial^{\alpha_1}_{\xb_1}\partial^{\alpha_2}_{\xb_2}\dots \partial^{\alpha_d}_{\xb_d}$ w.r.t. multi-indices $\alpha \in \Nbb_0^d$ with $0 \leq |\alpha| \leq k$.
If $T$ is chosen small enough such that $||t\Vb||_{W^{1,\infty}(\nlDom, \Rd)} < 1$, then the function $\Fbt^{\Vb}:\overline{\nlDom} \rightarrow \overline{\nlDom}$ is a bijective function, where $\Fbt^{\Vb}$ as well as the corresponding inverse $(\Fbt^{\Vb})^{-1}$ are both Lipschitz continuous (see \cite{henrot_shape}).
\begin{remark}
	In this work, the nonlocal boundary $\nlBound$ will not be deformed since the interface $\shape$ is supposed to be a subset of $\nlDom$. Thus, we extend every vector field $\Vb \in C_0^k(\nlDom,\Rd)$ by zero to a function defined on the domain $\nlDom \cup \nlBound$, i.e., we only consider vector fields that are elements of the set
	\begin{align*}
		\vecfields \defas \{\Vb:\nlDom \cup \nlBound \rightarrow \Rd | \Vb = 0 \textit{ on } \nlBound \textit{ and }  \left. \Vb \right|_{\overline{\nlDom}} \in C_0^k(\nlDom,\Rd) \}.
	\end{align*}
	Furthermore, $\Vb \in \vecfieldsspecific$ is sufficient for the calculations in this work.
\end{remark}~\\
Now, we continue by formulating the definition of the first and second shape derivative.
\begin{definition}[First Shape Derivative]
		Let $\shapeSpace$ be a suitable shape space. Then, for a shape functional $\objFun:\shapeSpace \rightarrow \R$, we define the Eulerian or directional derivative as
		\begin{align*}
			D_{\shape} \objFun(\shape)[\Vb] \defas \lim_{t \searrow 0} \frac{\objFun(\Fbt^{\Vb}(\shape)) - \objFun(\shape)}{t}.
		\end{align*}
		If $D_{\shape} \objFun(\shape)[\Vb]$ exists for all vector fields $\Vb \in \vecfieldsspecific$ and if the function
		\begin{align*}
			D_{\shape}\objFun(\shape) :\vecfieldsspecific \rightarrow \R,~ \Vb \mapsto D_{\shape}\objFun(\shape)[\Vb]
		\end{align*}
		is linear and continuous, the we refer to $D_{\shape}\objFun(\shape)[\Vb]$ as the shape derivative at $\shape$ in direction $\Vb$.
	\end{definition}
	\begin{definition}[Second Shape Derivative]
		\label{def:second_shape_der}
		If for all $\Vb,\Wb \in \vecfieldsspecific$ the derivative
		\begin{align*}
			D_{\shape}^2\objFun(\shape)[\Vb, \Wb] \defas \lim_{t \searrow 0} \frac{D_{\shape} \objFun(\Fbt^{\Wb}(\shape))[\Vb] - D_{\shape} \objFun(\shape)[\Vb]}{t} \text{ exists}
		\end{align*}
	and the function
	\begin{align*}
		D_{\shape}^2 \objFun(\shape): \vecfieldsspecific \times \vecfieldsspecific \rightarrow \R, \left(\Vb, \Wb\right) \mapsto D_{\shape}^2 \objFun(\shape)[\Vb, \Wb]
	\end{align*}
	is bilinear and continuous, then $D_{\shape}^2\objFun(\shape)[\Vb,\Wb]$ is called the second shape derivative at $\shape$ in the direction $(\Vb,\Wb)$. 
	\end{definition}~\\
	The second shape derivative can also be characterized as 
	\begin{align*}
		D_{\shape}^2 \objFun(\shape)[\Vb,\Wb] = \left.\frac{d}{dt} \right|_{t_2 = 0^+} \left. \frac{d}{dt}\right|_{t_1 = 0^+} \objFun\left(\Fb_{t_1}^{\Vb}(\Fb_{t_2}^{\Wb}(\shape))\right),
	\end{align*}
	which indicates, that the second shape derivative is computed by differentiating a functional, where the shape $\shape$ is first perturbed along the vector field $\Wb$ and afterwards along $\Vb$ as described in Definition \ref{def:pert_id}. 
	\begin{remark}
		In the literature, the second shape derivative $D_{\shape}^2 \objFun(\shape)$ is sometimes referred to as the shape Hessian at $\shape$.
	\end{remark}
	\begin{definition}[Shape Space]
		Given a domain $\initialShape \subset \nlDom$, we define the corresponding shape space $\shapeSpace(\initialShape)$ as
		\begin{align*}
			\shapeSpace(\initialShape) \defas \{\Tb(\initialShape)| \Tb: \overline{\nlDom} \rightarrow \overline{\nlDom} \textit{ invertible} \}.
		\end{align*}
	\end{definition}~\\
	In this work, as mentioned before, we only employ the perturbation of identity, which is invertible, if $t > 0$ is small enough. As shown in \cite{linear_view}, interpreting shape optimization as optimizing on a subset of $\{\Tb | \Tb :\overline{\nlDom} \rightarrow \overline{\nlDom}\}$ yields the linear second shape derivative
	\begin{align*}
		\objFun''(\shape)[\Vb, \Wb] \defas \left.\frac{d}{dt} \right|_{t_2 = 0^+} \left. \frac{d}{dt}\right|_{t_1 = 0^+} \objFun((\Id + t_2\Wb + t_1\Vb)(\shape)),
	\end{align*}
	which is symmetric, provides the basis of a Taylor series and thus justifies convergence properties of shape Newton methods. In \cite{structure_shape_der,sturm_newton}, it is proven under natural assumptions on the differentiability of \objFun, that the classical shape derivative can be expressed as
	\begin{align}
		\label{eq:structure_second_shape}
		D^2_{\shape} \objFun(\shape)[\Vb, \Wb] = \objFun''(\shape)[\Vb, \Wb] + D_{\shape} \objFun(\shape)[D\Vb\Wb].
	\end{align}
Here, the second term $D_{\shape} \objFun(\shape)[D\Vb\Wb]$ is only symmetric, if $D\Vb\Wb =D\Wb\Vb$ holds. In the next chapters, we first derive the classical shape derivative $D^2_{\shape} \objFun(\shape)[\Vb, \Wb]$ by making use of the so-called averaged adjoint method. By omitting the possibly nonsymmetric part $D_{\shape} \objFun(\shape)[D\Vb\Wb]$, we will then obtain the linear second shape derivative, which we utilize in a Newton-like algorithm to solve the constrained interface identification problem \eqref{nlDirInt}.
	\begin{remark}\label{remark:perimeter_der}
		The first and second shape derivative of the perimeter regularization $\perimeter$ can be formulated as
		\begin{align*}
			D_{\shape} \perimeter(\shape)[\Vb] &= \int_{\shape} \di_{\shape} (\Vb) ~ds = \int_{\shape} \di (\Vb) - \normal^{\top} D\Vb \normal ~ds \quad \text{and} \\
			D_{\shape}^2 \perimeter(\shape)[\Vb, \Wb] &= \int_{\shape} \di_{\shape} \Vb \di_{\shape} \Wb - \trace\left(D \Vb D \Wb \right) + \langle \left(D \Vb \right)^{\top} \normal, \left(D \Wb \right)^{\top} \normal \rangle ~ds,
		\end{align*}
		where $\di_{\shape} (\Vb) \defas \di (\Vb) - \normal^{\top} D\Vb \normal$ is the tangential divergence and $D_{\shape} \Vb \defas D \Vb - D\Vb \normal \normal^{\top}$ the tangential Jacobian. The formula for $D_{\shape} \perimeter(\shape)[\Vb]$ follows directly from \cite[Lemma 3.3.4]{schmidt_diss}. The derivation of $D_{\shape}^2 \perimeter(\shape)[\Vb, \Wb]$ can be found in \cite[Section 3.1]{schmidt2022capillary}. For ease of presentation, we exclude $\perimeter(\shape)$, $D_{\shape} \perimeter(\shape)[\Vb]$ and $D_{\shape}^2 \perimeter(\shape)[\Vb, \Wb]$ in the calculations of the next chapter.
	\end{remark}
\begin{remark}
	Given a vector field $\Vb \in \vecfieldsspecific$, we set $\shapet^{\Vb}\defas \FbtV(\shape)$ and denote by $\nlDom(\shapet^{\Vb})$ that we consider $\nlDom$ to be decomposed in two domains by $\shapet^{\Vb}$, i.e., $\nlDom(\shapet^{\Vb})=\FbtV(\nlDom_1) \dot{\cup} \shapet^{\Vb} \dot{\cup} \FbtV(\nlDom_2)$ and $\shapet^{\Vb} = \partial \FbtV(\nlDom_1) \cap \partial \FbtV(\nlDom_2)$.
\end{remark}
\section{Second Shape Derivative via the Averaged Adjoint Method}
\label{chap:AAM_second_der}
As mentioned earlier, the interface identification problem constrained by nonlocal models \eqref{nlDirInt} can be written as a reduced optimization problem
\begin{align} \label{def:red_prob}
		\min_{\shape \in \shapeSpace(\initialShape)} \redFun(\shape) \defas \objFun(\weakSol(\shape), \shape),
\end{align}
where $\weakSol(\shape)$ solves \eqref{eq:weak_form_int}.
In \cite{shape_paper} it was shown, how the first shape derivative of \eqref{def:red_prob} can be computed via the averaged adjoint method\cite{sturm_minimax, AAM, Sturm_diss}. Therefore, we refer for the first derivative to \cite{shape_paper} and continue by introducing the \emph{averaged adjoint method for systems of state equations} of \cite{Sturm_diss}, which we apply to develop a formula for the second shape derivative of \eqref{def:red_prob}. We can express the first shape derivative of the reduced functional at the shape $\FbtW(\shape)=\shapetW$ in the direction $\Vb$ as
\begin{align*}
	&D_{\shape}\redFun(\shapetW)[\Vb] = \objFunDertW(\weakSol(\shapetW)) - \varForceDertW(\advar(\shapetW)) + \varOpDertW(\weakSol(\shapetW), \advar(\shapetW)),
\end{align*}
where the three components are defined below.
Here, $\weakSol(\shapetW)$ and $\advar(\shapetW)$ satisfy the equations
\begin{align*}
	\intOptW(\weakSol(\shapetW), \psi) = \intForcetW(\psi) \text{ for all } \psi \in \testSpace(\Fbt^{\Wb}(\nlDom) \cup \nlBound) \text{ and} \\
	\intOptW(\varphi, \advar(\shapetW)) = \secondIntForcetW(\weakSol(\shapetW), \varphi) \text{ for all } \varphi \in \testSpace(\Fbt^{\Wb}(\nlDom) \cup \nlBound),
\end{align*}
where $\secondIntForcetW(\weakSol, \varphi) \defas -\int_{\Fbt^{\Wb}(\nlDom)} \left(\weakSol - \data \right)\varphi ~d\xb$. Additionally, we set $\Vb(\xb,\yb) \defas \left(\Vb(\xb), \Vb(\yb)\right)^{\top}$ as well as
\begin{align} 
	\Psi_{\shapetW,\Vb}^1(\xb,\yb)\defas\grad \kernel_{\shapetW}(\xb,\yb)^{\top}\Vb(\xb,\yb) \text{ and }
	&\Psi_{\shapetW,\Vb}^2(\xb,\yb) \defas \kernel_{\shapetW}(\xb,\yb)(\di \Vb(\xb) + \di \Vb(\yb)), \label{defs:Phis}
\end{align}
where $\grad \kernel_{\shapetW} = \left(\grad_{\xb}\kernel_{\shapetW}, \grad_{\yb}\kernel_{\shapetW}\right)^{\top}$.
Thus, the three shape derivatives of objective, forcing term and bilinear form are defined in the following way:
\begin{align*}
	&\objFunDertW(\weakSol) \defas \int_{\Fbt^{\Wb}(\nlDom)} -(\weakSol(\xb) - \data(\xb))\grad\data(\xb)^{\top}\Vb(\xb) + \frac{1}{2} (\weakSol(\xb) - \data(\xb))^2 \di \Vb(\xb) ~d\xb, \\
	&\varForceDertW(\advar)\defas \int_{\Fbt^{\Wb}(\nlDom)} \advar(\xb) \grad f(\xb)^{\top}\Vb(\xb) + f(\xb)\advar(\xb) \di \Vb(\xb) ~d\xb \text{ and} \\
	&\varOpDertW(\weakSol, \advar) \defas \frac{1}{2} \iint_{(\Fbt^{\Wb}(\completeDom))^2} \left(\advar(\xb) - \advar(\yb)\right)\left(\weakSol(\xb)\Psi_{\shapetW,\Vb}^1(\xb,\yb) - \weakSol(\yb)\Psi_{\shapetW,\Vb}^1(\yb,\xb) \right) ~d\yb d\xb\\
	&+\frac{1}{2} \iint_{(\Fbt^{\Wb}(\completeDom))^2} \left(\advar(\xb) - \advar(\yb)\right)\left(\weakSol(\xb)\Psi_{\shapetW,\Vb}^2(\xb,\yb) - \weakSol(\yb)\Psi_{\shapetW,\Vb}^2(\yb,\xb)\right) ~d\yb d\xb.
\end{align*}
Then, we set the corresponding Lagrangian as
\begin{align*}
	\lagrangian(t, \weakSol, \advar, \psi, \varphi) \defas \objFunDertW(\weakSol) - \varForceDertW(\advar) + \varOpDertW(\weakSol, \advar) &+ \intOptW(\weakSol,\psi) - \intForcetW(\psi) \\
	&+ \intOptW(\varphi, \advar) - \secondIntForcetW(\weakSol, \varphi).
\end{align*}
Thus, for the reduced functional holds
\begin{align*}
	D_{\shape}\redFun(\shapetW)[\Vb] = \lagrangian(t,\weakSol(\shapetW),\advar(\shapetW), \psi, \varphi) \text{ for any } \psi,\varphi \in \testSpace(\nlDom(\shapetW) \cup \nlBound).
\end{align*}
In order to avoid computing derivatives of $\weakSol$, $\advar$, $\psi$ or $\varphi$ to derive the second shape derivative, we use the perturbation of identity, which is a diffeomorphism, as a so-called \emph{pull-back function}. Due to this characteristic of the perturbation of identity, we can find for any function ${\weakSol \in \testSpace(\nlDom(\shapetW) \cup \nlBound)}$ a unique $\widetilde{\weakSol} \in \testSpace(\completeDom)$ with $\weakSol = \widetilde{\weakSol} \circ (\Fbt^{\Wb})^{-1}$.
Consequently, we can define an alternative Lagrangian functional as
\begin{align*}
	&\reallagrangian:[0,T] \times \testSpace(\completeDom)^4 \rightarrow \R, \\
	&\reallagrangian(t, \weakSol, \advar, \psi, \varphi) \defas \lagrangian(t, \weakSol \circ (\Fbt^{\Wb})^{-1}, \advar \circ (\Fbt^{\Wb})^{-1}, \psi \circ (\Fbt^{\Wb})^{-1}, \varphi \circ (\Fbt^{\Wb})^{-1}).
\end{align*}
Another advantage of this function space parametrization is, that the function spaces for $\weakSol$, $\advar$, $\psi$ and $\varphi$ do not depend on $t$ anymore. Further, we set
\begin{align*}
	&\objFunDer: [0,T] \times \testSpace(\completeDom) \rightarrow \R, ~\objFunDer(t, \weakSol) \defas \objFunDertW(\weakSol \circ (\Fbt^{\Wb})^{-1}),\\
	&\varForceDer: [0,T] \times \testSpace(\completeDom) \rightarrow \R, ~\varForceDer(t, \weakSol) \defas \varForceDertW(\advar \circ (\Fbt^{\Wb})^{-1}),\\
	&\varOpDer: [0,T] \times \testSpace(\completeDom)^2 \rightarrow \R, ~\varOpDer(t, \weakSol, \advar) \defas \varOpDertW(\weakSol \circ (\Fbt^{\Wb})^{-1}, \advar \circ (\Fbt^{\Wb})^{-1}),\\
	&\varOp: [0,T] \times \testSpace(\completeDom)^2 \rightarrow \R, ~\varOp(t, \weakSol, \advar) \defas \intOptW(\weakSol \circ (\Fbt^{\Wb})^{-1}, \advar \circ (\Fbt^{\Wb})^{-1}),\\
	&\varForce:[0,T] \times \testSpace(\completeDom) \rightarrow \R, ~\varForce(t,\psi) \defas \intForcetW(\psi \circ (\Fbt^{\Wb})^{-1}),\\
	&\secondVarForce:[0,T] \times \testSpace(\completeDom)^2 \rightarrow \R, ~\secondVarForce(t, \weakSol, \varphi) \defas \secondIntForcetW(\weakSol \circ (\Fbt^{\Wb})^{-1}, \varphi \circ (\Fbt^{\Wb})^{-1}).
\end{align*}
 In order to apply AAM the following assumptions need to be satisfied.
\begin{itemize}
	\item[\bf(D0)] For all $t \in [0, T]$ and $\weakSol,\widetilde{\weakSol}, \widehat{\weakSol}, \advar, \widetilde{\advar}, \widehat{\advar}, \psi, \varphi \in \testSpace(\completeDom)$ we have
	\begin{enumerate}[(i)]
		\item $\reallagrangian_{1}: [0,1] \rightarrow \R, s \mapsto \reallagrangian(t, \weakSol + s\widetilde{\weakSol}, \advar, \psi, \varphi) \text{ and } 
		\reallagrangian_{2}: [0,1] \rightarrow \R, s \mapsto \reallagrangian(t, \weakSol, \advar + s\widetilde{\advar}, \psi, \varphi)$ are absolutely continuous.
		\item $[0,1]\ni s \mapsto d_{\weakSol} \reallagrangian(t, \weakSol + s\widehat{\weakSol}, \advar, \psi, \varphi)[\widetilde{\weakSol}] \text{ and } [0,1] \ni s \mapsto d_{\advar}\reallagrangian(t, \weakSol, \advar + s\widehat{\advar}, \psi, \varphi)[\widetilde{\advar}]$ belong to $L^1((0,1))$.
		\item $\psi \mapsto \reallagrangian(t, \weakSol, \advar, \psi, \varphi) \text{ and } \varphi \mapsto \reallagrangian(t, \weakSol, \advar, \psi, \varphi)$ are affine-linear.
	\end{enumerate}
	\item[\bf(D1)] For all $\weakSol,\advar,\psi,\varphi \in \testSpace(\completeDom)$ the function
	\begin{align*}
		[0,T] \rightarrow \R, t \mapsto \reallagrangian(t,\weakSol,\advar,\psi,\varphi)
	\end{align*}
is differentiable.
	\item[\bf(D2)] For all $t \in [0,T]$
	\begin{itemize}
		\item there exist unique solutions $\weakSol^t$ and $\advar^t$ to the (state) equations
		\begin{align}
			\varOp(t, \weakSol^t,\psi) = \varForce(t,\psi) \text{ for all } \psi \in \testSpace(\completeDom), \label{eq:state_1} \\
			\varOp(t, \varphi, \advar^t) = \secondVarForce(t, \weakSol^t, \varphi) \text{ for all } \varphi \in \testSpace(\completeDom). \label{eq:state_2}
		\end{align}
		\item there exist unique solutions $\psi^t, \varphi^t \in \testSpace(\completeDom)$ to the \emph{averaged adjoint equations}
		\begin{align}
			\int_0^1 d_{\weakSol} \reallagrangian(t, s\weakSol^t + (1-s)\weakSol^0,\advar^t, \psi^t,\varphi^t)[\widetilde{\weakSol}] ~ds = 0 \text{ for all } \widetilde{\weakSol} \in \testSpace(\completeDom), \label{eq:AAE1} \\
			\int_0^1 d_{\advar} \reallagrangian(t, \weakSol^0, s\advar^t + (1-s)\advar^0,\psi^t,\varphi^t)[\widetilde{\advar}] ~ds = 0 \text{ for all } \widetilde{\advar} \in \testSpace(\completeDom). \label{eq:AAE2}
		\end{align}
	\end{itemize}
	\item[\bf(D3)] For every sequence $(s_n)_{n \in \Nbb}$ with $\lim_{n \rightarrow \infty} s_n = 0$, there exists a subsequence $(s_{n_k})_{k \in \Nbb}$ such that
	\begin{align*}
		\lim_{k \rightarrow \infty,
		t \searrow 0} \partial_t \reallagrangian(t, \weakSol^0, \advar^0, \psi^{s_{n_k}}, \varphi^{s_{n_k}}) = \partial_t \reallagrangian(0,\weakSol^0,\advar^0,\psi^0,\varphi^0).
	\end{align*}
\end{itemize}
\begin{theorem}[Averaged Adjoint Method for Systems of State Equations]
If the assumptions (D0)-(D3) are fulfilled, we get for any $\psi,\varphi \in \testSpace(\completeDom)$ that
\begin{align*}
	\left. \frac{d}{dt} \right|_{t\searrow 0} \reallagrangian(t,\weakSol^t,\advar^t,\psi,\varphi) = \partial_t \reallagrangian(0,\weakSol^0,\advar^0,\psi^0,\varphi^0).
\end{align*}
\end{theorem}
\begin{proof}
	See \cite[Theorem 4.5]{Sturm_diss}.
\end{proof}
\begin{corollary}
	If assumptions (D0)-(D3) hold, the second shape derivative of the reduced functional can be computed via
	\begin{align*}
		D_{\shape}^2\redFun(\shape)[\Vb,\Wb] = \partial_t \reallagrangian(0,\weakSol^0,\advar^0,\psi^0,\varphi^0).
	\end{align*} 
\end{corollary}~\\
Finally, we would like to highlight the fact, that the solutions $\weakSol^0$ and $\advar^0$ to the state equations \eqref{eq:state_1} and \eqref{eq:state_2} for $t=0$ are independent of $\Vb$. However, since the adjoint equations $\eqref{eq:AAE1}$ and $\eqref{eq:AAE2}$ involve the shape derivative at $\shape$ in the direction $\Vb$, the functions $\psi^0$ and $\varphi^0$ need to be computed for every $\Vb$ separately. Therefore, we write $\psi(\shape,\Vb)$ and $\varphi(\shape,\Vb)$ in Chapter \ref{chap:num_exp} to explicitly express which adjoints are calculated during the presented algorithm.
\section{Deriving the Second Shape Derivative for the Reduced Objective Function}
\label{chap:AAM_proof}
Before we actually calculate the second shape derivative, we need to prove that the Assumptions (D0)-(D3) hold.
\begin{remark}
	In the following we use the abbreviations 
	\begin{align*}
		&\data^t(\xb) \defas \data(\FbtW(\xb)), f^t(\xb) \defas f_{\shapetW}(\FbtW(\xb)), \kernelt(\xb, \yb) \defas \kernel_{\shapetW}(\FbtW(\xb), \FbtW(\yb)), \xt(\xb) \defas D\FbtW(\xb),\\
		&\Vbt(\xb) \defas \Vb(\FbtW(\xb)), \Vbt(\xb,\yb) \defas (\Vbt(\xb), \Vbt(\yb))^{\top}, \Psi_1^t(\xb,\yb)\defas \grad \kernelt(\xb,\yb)^{\top}\Vbt(\xb,\yb) \text{ and} \\
		&\Psi_2^t(\xb,\yb) \defas \kernelt(\xb,\yb)(\di \Vbt(\xb) + \di \Vbt(\yb)).
	\end{align*}
As a result, $\Psi_i^t(\xb,\yb) = \Psi_{\shapetW,\Vbt}^i(\FbtW(\xb),\FbtW(\yb))$ for all $\xb,\yb \in \completeDom$ and $i=1,2$.
\end{remark}
\begin{lemma}\label{lemma:D0_fulfilled}
	Given any $t \in [0, T]$ and $\weakSol,\widetilde{\weakSol}, \widehat{\weakSol}, \advar, \widetilde{\advar}, \widehat{\advar}, \psi, \varphi \in \testSpace(\completeDom)$, Assumption (D0) is fulfilled.
\end{lemma}
\begin{proof}
	\begin{enumerate}
		\item We get by using the linearity of $\varOp$ and $\varOpDer$ regarding the second and third argument
	\begin{align*}
		\reallagrangian_{1}'(s) =& d_\weakSol \reallagrangian(t, \weakSol + s\widetilde{\weakSol}, \advar, \psi, \varphi)[\widetilde{\weakSol}] \\
		=& d_{\weakSol} \objFunDer(t, \weakSol + s \widetilde{\weakSol})[\widetilde{\weakSol}]
		+ \varOpDer(t, \widetilde{\weakSol}, \advar)
		+ \varOp(t, \widetilde{\weakSol}, \psi) - d_{\weakSol}\secondVarForce(t, \weakSol + s\widetilde{\weakSol}, \varphi)[\widetilde{\weakSol}], \\
		\reallagrangian_{2}'(s) =& d_\advar \reallagrangian(t, \weakSol, \advar + s\widetilde{\advar}, \psi, \varphi)[\widetilde{\advar}]
		= - \varForceDer(t,\widetilde{\advar}) + \varOpDer(t, \weakSol, \widetilde{\advar}) + \varOp(t,\varphi, \widetilde{\advar}),
	\end{align*}
		where $\varForceDer$ is also linear in the second argument and where we have
		\begin{align*}
			d_{\weakSol}\secondVarForce(t,\weakSol, \varphi)[\widetilde{\weakSol}] &= -\int_{\nlDom} \widetilde{\weakSol}\varphi\xt~d\xb \text{ and } \\
			d_{\weakSol} \objFunDer(t,\weakSol)[\widetilde{\weakSol}] &= \int_{\nlDom} - \widetilde{\weakSol} (\grad \data^t)^{\top} \Vb^t \xt  + (\weakSol - \data^t) \widetilde{\weakSol} \di \Vb^t \xt ~d\xb.
		\end{align*}
		Thus, $\reallagrangian_{1}$ and $\reallagrangian_{2}$ are continuously differentiable and therefore absolutely continuous. Lastly, note that $d_{\weakSol}\secondVarForce(t,\weakSol, \varphi)[\widetilde{\weakSol}]$ is independent of the choice of $\weakSol$.
		\item By examining the computations in 1. we directly see that $\reallagrangian_{2}'$ is constant on $(0,1)$. Further, the variable $s$ can be found in only one integral of $\reallagrangian_{1}'$ and we get as a consequence of Fubini's theorem
		\begin{align*}
			\int_0^1 d_{\weakSol}\reallagrangian_{1}(t, \weakSol + s \widehat{\weakSol},\advar,\psi, \varphi)[\widetilde{\weakSol}] ~ds =& \int_{\nlDom} \left(\weakSol + \frac{1}{2} \widehat{\weakSol} - \data^t \right) \widetilde{\weakSol}\di \Vb^t \xt - \widetilde{\weakSol}(\grad\data^t)^{\top}\Vb^t \xt ~d\xb \\
			&+ \varOpDer(t,\widetilde{\weakSol}, \advar) + \varOp(t,\widetilde{\weakSol}, \psi) - d_{\weakSol}\secondVarForce(t,\weakSol, \advar)[\widetilde{\weakSol}].
		\end{align*}
		\item Direct consequence due to the linearity of $\varOp$, $\varForce$ and $\secondVarForce$ in the corresponding argument.
	\end{enumerate}
\end{proof}
\begin{corollary}
	As seen in the proof of Lemma \ref{lemma:D0_fulfilled}, the averaged adjoint equations \eqref{eq:AAE1} and \eqref{eq:AAE2} can be formulated as
	\begin{align}
		\text{Find } \psi^t,\varphi^t &\in \testSpace(\completeDom) \text{ such that} \nonumber\\
		\varOp(t, \widetilde{\weakSol}, \psi^t) =& - d_{\weakSol} \objFunDer(t, \frac{1}{2}(\weakSol^t + \weakSol^0))[\widetilde{\weakSol}] - \varOpDer(t, \widetilde{\weakSol}, \advar^t) + d_{\weakSol} \secondVarForce(t, \weakSol^0,\varphi^t)[\widetilde{\weakSol}]  \label{eq:AAE1_nonlocal} \\
		&\hspace{15em} \text{for all } \widetilde{\weakSol} \in \testSpace(\completeDom) \text{ and} \nonumber\\
	\varOp(t,\varphi^t, \widetilde{\advar}) =& \varForceDer(t,\widetilde{\advar}) - \varOpDer(t,\weakSol^0,\widetilde{\advar}) \text{ for all } \widetilde{\advar} \in \testSpace(\completeDom). \label{eq:AAE2_nonlocal}	
\end{align}
\end{corollary}
\begin{lemma}[{After \cite[Proposition 2.32]{sokolowski_Introduction}}]
	\label{lemma:frechet_diff_bounded_domain}
	Let $\holdAll \subset \Rd$ be an open and bounded domain with nonzero measure, $f \in W^{1,1}(\holdAll,\R)$ and $\Vb \in C_{0}^1(\holdAll, \Rd)$, where $k \in \Nbb$. Moreover, assume that $T > 0$ is sufficiently small, such that $\FbtV: \overline{\holdAll} \rightarrow \overline{\holdAll}$ is bijective for all $t \in [0,T]$. Then, $t \mapsto f \circ \FbtV$ is differentiable in $L^1(\holdAll)$ with
	\begin{align*}
		\left.\frac{d}{dt}\right|_{t=r} \left(f \circ \FbtV\right)(\xb) = \grad f(\textbf{F}_r^{\Vb}(\xb) )^{\top} \Vb(\xb) \text{ for } r \in [0,T].
	\end{align*} 
\end{lemma}
\begin{proof}
	See Appendix \ref{app:proof_frechet_diff}.
\end{proof}~\\
\textbf{Assumption (S0):}
\begin{itemize}
	\item For every $t \in [0, T]$, there exist unique solutions $\weakSol^t, \advar^t \in \testSpace(\completeDom)$ to the state equations \eqref{eq:state_1} and \eqref{eq:state_2}. Moreover, there are unique functions $\psi^t, \varphi^t$ that solve \eqref{eq:AAE1_nonlocal} and \eqref{eq:AAE2_nonlocal}, i.e., Assumption (D2) is fulfilled.
	\item There exists a constant $C_* > 0$ such that
	\begin{align}\label{eq:S0_coercivity}
		\varOp(t, \weakSol, \weakSol) \geq C_*||\weakSol||^2_{L^2(\completeDom)} \text{ for all } \weakSol \in \testSpace(\completeDom) \text{ and } t \in [0,T].
	\end{align}
\end{itemize}
\textbf{Assumption (S1):}\\
Each of the two kernel classes has to fulfill additional conditions:
\begin{itemize}
	\item For integrable kernels it is assumed to hold 
	\begin{align*}
		\kernel_{ij} \in W^{2,\infty}(\nlDom \times \nlDom,\R),~ \kernel_{i\nlBound} \in W^{2,\infty}((\nlDom \times \nlBound) \cup (\nlBound \times \nlDom),\R)  \text{ for all } i,j=1,2.
	\end{align*}
	\item For $n \in \Nbb$, define 
	\begin{align*}
		D^{\nlDom}_n &\defas \{(\xb,\yb) \in \nlDom \times \nlDom: ||\xb - \yb||_2 \geq \frac{1}{n} \} \text{ and}  \\ 
		D_{n}^{\nlBound} &\defas \{(\xb,\yb) \in (\nlDom \times \nlBound) \cup (\nlBound \times \nlDom): ||\xb - \yb||_2 \geq \frac{1}{n} \}.
	\end{align*}
 	Further, we set $D_n \defas D_n^{\nlDom} \cup D_n^{\nlBound}$. Then, for a singular symmetric kernel we assume ${\kernel_{ij} \in W^{2,\infty}(D_n^{\nlDom},\R)}$ and ${\kerneliI \in W^{2,\infty}(D_n^{\nlBound}, \R)}$. Additionally, we suppose that
	\begin{align*}
		&|\kernelij(\xb,\yb)|||\xb - \yb||_2^{d+2s} \in L^{\infty}(\nlDom \times \nlDom),~ |\kerneliI(\xb,\yb)|||\xb - \yb||_2^{d+2s} \in L^{\infty}((\nlDom \times \nlBound)\cup(\nlBound \times \nlDom)),\\
		&|\grad \kernelij(\xb,\yb)^{\top}\Vb(\xb,\yb)|||\xb - \yb||_2^{d+2s} \in L^{\infty}(\nlDom \times \nlDom), \\
		&|\grad_{\xb} \kerneliI(\xb,\yb)^{\top}\Vb(\xb)|||\xb - \yb||_2^{d+2s} \in L^{\infty}((\nlDom \times \nlBound)\cup(\nlBound \times \nlDom)),\\
		&|\Vb(\xb,\yb)^{\top}\hess(\kernelij)(\xb,\yb)\Wb(\xb,\yb)|||\xb - \yb||_2^{d+2s} \in L^{\infty}(\nlDom \times \nlDom) \text{ and}\\ &|\Vb(\xb)\frac{d^2}{d\xb d\xb} \kerneliI(\xb,\yb)\Wb(\xb)|||\xb - \yb||_2^{d+2s} \in L^{\infty}((\nlDom \times \nlBound)\cup(\nlBound \times \nlDom)), 
	\end{align*}
	hold for $i,j =1,2$ and for all $\Vb,\Wb \in C_0^1(\completeDom, \Rd)$.
\end{itemize}
As we will see in Lemma \ref{lemma:D1_fulfilled}, Assumption (S1) yields the well-posedness of certain integrals, which are components of the second shape derivative. Moreover, the fact that vector fields $\Vb$ and $\Wb$ vanish on $\nlBound$ is already incorporated into the conditions w.r.t. $\kerneliI$ in  (S1), i.e., we only have to differentiate regarding $\xb$. 
\begin{example}
\begin{sloppypar}
Let a symmetric function $\sigma_{\shape} \defas \sum_{i,j=1,2} \sigma_{ij} \ind_{\nlDom_i \times \nlDom_j} + \sum_{i=1,2} \sigma_{iI} \ind_{(\nlDom_i \times \nlBound) \cup (\nlBound \times \nlDom_i)}$ be given, where
$\sigma_{ij}:\Rd \times \Rd \rightarrow [0,\infty)$ and $\sigma_{iI}:\Rd \times \Rd \rightarrow [0,\infty)$ for $i,j=1,2$. Here, the functions $\sigma_{ii}$ as well as $\sigma_{i\nlBound}$ are supposed to be symmetric for $i=1,2$ and we consider ${\sigma_{12}(\xb,\yb)=\sigma_{21}(\yb,\xb)}$ to hold. Additionally, we assume the existence of constants ${0 < \sigma_* \leq \sigma^* < \infty}$ such that
${\sigma_* \leq \sigma_{ij} \leq \sigma^*}$ on $\nlDom \times \nlDom$ and ${\sigma_* \leq \sigma_{iI} \leq \sigma^*}$ on $(\nlDom \times \nlBound) \cup (\nlBound \times \nlDom)$ for all $i,j=1,2$. Then, a popular choice of a singular symmetric kernel is the function
\end{sloppypar}
\begin{align*}
	\kernel_{\shape}(\xb,\yb) \defas \frac{\sigma_{\shape}(\xb,\yb)}{||\xb-\yb||_2^{d+2s}}\ind_{B_{\delta}(\xb)}(\yb).
\end{align*}
Due to the boundedness of $\sigma_{ij}$ and $\sigma_{i\nlBound}$ from below and above as described before, Assumption (S0) is already fulfilled, since condition \eqref{ineq:sing_sym_kernel} also holds for perturbed domains $\nlDom(\shapetW) \cup \nlBound$.
We now show that Assumption (S1) is also valid, if $\sigma_{ij}, \sigma_{i\nlBound} \in W^{2,\infty}(\Rd \times \Rd ,\R)$. Therefore, we set $\widetilde{\kernel}(\xb,\yb) \defas \frac{1}{||\xb - \yb||_2^{d+2s}}$, $c_1 \defas -(d+2s)$ and $c_2 \defas 2(d+2s) + (d+2s)^2$. For $\widetilde{\kernel}$ we derive
\begin{align*}
	\frac{d}{d \xb} \widetilde{\kernel}(\xb,\yb) = c_1 \widetilde{\kernel}(\xb,\yb) \frac{\xb - \yb}{||\xb - \yb||_2^{2}} \text{ and } \frac{d}{d \yb} \widetilde{\kernel}(\xb,\yb) = - \frac{d}{d \xb} \widetilde{\kernel}(\xb,\yb). 
\end{align*}
Moreover, the second partial derivatives of $\widetilde{\kernel}$ can be expressed as
\begin{align*}
	&\frac{d^2}{d\xb d\xb} \widetilde{\kernel}(\xb,\yb) = c_1 \widetilde{\kernel}(\xb,\yb) \frac{\Id}{||\xb-\yb||_2^{2}} + c_2\widetilde{\kernel}(\xb,\yb)\frac{(\xb - \yb)(\xb - \yb)^{\top}}{||\xb -\yb||_2^4},\\
	&\frac{d^2}{d\xb d\yb} \widetilde{\kernel}(\xb,\yb)=\frac{d^2}{d\yb d\xb} \widetilde{\kernel}(\xb,\yb)=-\frac{d^2}{d\xb d\xb} \widetilde{\kernel}(\xb,\yb) \text{ and } \frac{d^2}{d\yb d\yb} \widetilde{\kernel}(\xb,\yb) = \frac{d^2}{d\xb d\xb} \widetilde{\kernel}(\xb,\yb).
\end{align*}
Since $\kernel_{ij} = \sigma_{ij} \widetilde{\kernel}$ and due to the Lipschitz continuity of $\Vb$ and $\Wb$, we can conclude
\begin{align*}
	&|\grad \kernelij(\xb,\yb)^{\top}\Vb(\xb,\yb)|||\xb - \yb||_2^{d+2s} \\
	&\leq |c_1\sigma_{ij}(\xb,\yb)\frac{(\xb-\yb)^\top(\Vb(\xb) - \Vb(\yb))}{||\xb-\yb||_2^2}| + |\grad \sigma_{ij}(\xb,\yb)^{\top}\Vb(\xb,\yb)| < C < \infty \text{ and} \\
	&|\Vb(\xb,\yb)^{\top}\hess(\kernelij)(\xb,\yb)\Wb(\xb,\yb)|||\xb - \yb||_2^{d+2s} \\
	&= |(\Vb(\xb) - \Vb(\yb))\Biggl(\sigma_{ij}(\xb,\yb)\left(c_1\frac{\Id}{||\xb - \yb||_2^2} + c_2\frac{(\xb-\yb)(\xb-\yb)^{\top}}{||\xb-\yb||_2^{4}}\right) + 2\grad \sigma_{ij}(\xb,\yb) \frac{\xb-\yb}{||\xb-\yb||_2^2} \\
	&\hspace{8em} + \hess(\sigma_{ij})(\xb,\yb)\Biggr)(\Wb(\xb) - \Wb(\yb))| < C < \infty \text{ for a.e. } \xb,\yb\in\nlDom,
\end{align*}
for some constant $C>0$. The remaining conditions on $\kernel_{i\nlBound}$ also hold, which can be proven analogously.
\end{example}
\begin{remark}\label{remark:sing_kernel_bounded_Dn}
	As a consequence of \cite[Lemma 6.13]{allaire_conception}, if $T>0$ is small enough, there exists a Lipschitz constant $L>0$ such that
	\begin{align*}
		||\Fbt^{-1}(\xb) - \Fbt^{-1}(\yb)||_2 \leq \frac{1}{L} ||\xb - \yb ||_2 \text{ for } \xb,\yb \in \completeDom \text{ and } t \in [0, T].
	\end{align*}
	Let $\kernel_{\shape}$ be a singular symmetric kernel with $|\kernel_{ij}(\xb,\yb)|||\xb-\yb||_2^{d+2s} < \kernel^*$ for $\xb,\yb \in \nlDom$ and $|\kernel_{i\nlBound}(\xb,\yb)|||\xb-\yb||_2^{d+2s} < \kernel^*$ for $(\xb,\yb) \in (\nlDom\times\nlBound)\cup(\nlBound\times\nlDom)$ for some constant $\kernel^*>0$. Thus, $\kernelt$ is also Lipschitz continuous on $D_n$, since
	\begin{align*}
		\kernelt(\xb,\yb) = \kernel_{\shapetW}(\Fbt(\xb),\Fbt(\yb)) \leq \frac{\kernel^*}{||\Fbt(\xb)-\Fbt(\yb)||_2^{d+2s}} \leq \frac{L\kernel^*}{||\xb-\yb||_2^{d+2s}} \leq n L \kernel^* \text{ for } (\xb,\yb) \in D_n.
	\end{align*}
As a result, for all $t \in [0,T]$ and $n \in \Nbb$ the corresponding $\kernelt$ is also bounded on $D_n$.
Therefore, if Assumption (S1) is fulfilled, it can be shown in similar way, that all weak first and second partial derivatives of $\kernelt$ are essentially bounded on $D_n$ for all $t \in [0,T]$ and $n \in \Nbb$. 
\end{remark}
\begin{lemma}
	\label{lemma:D1_fulfilled}
	If Assumption (S1) is fulfilled, then for any $\weakSol, \advar, \psi, \varphi \in \testSpace(\completeDom)$ and $t \in [0, T]$ the derivative $\partial_t \reallagrangian(t, \weakSol, \advar,\psi, \varphi)$ exists, i.e., (D1) is satisfied. 
\end{lemma}
\begin{proof}
	For the proof of the differentiability of $\varOp(t,\weakSol,\psi)$, $\varOp(t,\varphi,\advar)$ and $\varForce(t,\psi)$ regarding $t$, we refer to \cite{shape_paper}. It remains to show that $\partial_t \secondVarForce(t,\weakSol,\varphi)$, $\partial_t \objFunDer(t, \weakSol)$, $\partial_t \varForceDer(t, \advar)$ and $\partial_t \varOpDer(t,\weakSol,\advar)$ exist. Because $\data \in H^2(\nlDom)$ and $\xt=\det D\FbtW$ is continuously differentiable as a composition of continuously differentiable functions, we get by applying Lemma \ref{lemma:frechet_diff_bounded_domain} and the product rule for Fr{\'e}chet derivatives that
	\begin{align*}
		\partial_t \secondVarForce(t, \weakSol,\varphi) = \int_{\nlDom} (\grad \data^t)^{\top} \Wb \varphi \xt - (\weakSol - \data^t) \varphi \left. \frac{d}{dr} \right|_{r=t} (\xi^r) ~d\xb.
	\end{align*} 
	Additionally, since $f^0 \in H^2(\nlDom)$ and $\Vb \in \vecfieldsspecific$, we can conclude by making use of Lemma \ref{lemma:frechet_diff_bounded_domain} and by applying the chain and product rule of differentiation that
	\begin{align*}
		\partial_t \objFunDer(t,\weakSol,\advar) &= \int_{\nlDom}  (\grad\data^t)^{\top}\Wb (\grad\data^t)^{\top}\Vbt \xt ~d\xb \\
		&- \int_{\nlDom} (\weakSol - \data^t)\left((\Vbt)^{\top} \hess(\data)^t\Wb\xt + (\grad \data^t)^{\top}D\Vb^t\Wb\xt + (\grad\data^t)^{\top}\Vbt \left. \frac{d}{dr} \right|_{r=t}(\xi^r) \right) ~d\xb \\
		&+ \int_{\nlDom}-(\weakSol - \data^t)(\grad\data^t)^{\top}\Wb \di \Vbt \xt \\
		&\quad\quad + \frac{1}{2} \left(\weakSol - \data^t\right)^2\left((\grad \di \Vbt)^{\top}\Wb \xt
		+ \di \Vbt \left. \frac{d}{dr} \right|_{r=t} \left( \xi^r \right)\right) ~d\xb \text{ and}\\
		\partial_t \varForceDer(t,\advar) &= \int_{\nlDom} (\Vbt)^{\top}\hess(f^t)\Wb \advar \xt + (\grad f^t)^{\top} D\Vb^t\Wb\advar\xt + (\grad f^t)^{\top}\Vbt \advar \left. \frac{d}{dr} \right|_{r=t} (\xi^r) ~d\xb\\
		&+\int_{\nlDom} (\grad f^t)^{\top} \Wb \advar \di \Vbt \xt + f^t\advar (\grad \di \Vbt)^{\top}\Wb \xt + f^t \advar \di \Vbt \left. \frac{d}{dr} \right|_{r=t} \left(\xi^r \right) ~d\xb. 
	\end{align*}
\begin{sloppypar}
\textbf{Integrable kernels:}\\
Since we assume $\kernel_{\shapetW} \in W^{2,\infty}((\completeDom)\times(\completeDom))$ as a consequence of Assumption (S1) and $\Vb \in \vecfieldsspecific$, we conclude that $\kernel_{\shapetW},\Psi_{\shapetW,\Vb}^1 \in W^{1,1}((\completeDom)^2)$ with ${\grad \Psi_{\shapetW,\Vb}^1=\hess(\kernel_{\shapetW})\Vb + (\grad\kernel_{\shapetW}^{\top} D\Vb)^{\top}}$. Additionally, $\xt$ and $\di \Vb$ are continuously differentiable and we get the weak derivative
\begin{align*}
	\grad \Psi_{\shapetW,\Vb}^2(\xb,\yb) = (\di \Vb(\xb) + \di \Vb(\yb))\grad\kernel_{\shapetW}(\xb,\yb) + \kernel_{\shapetW}(\xb,\yb)(\grad \di \Vb(\xb) + \grad \di \Vb(\yb)).
\end{align*}
\end{sloppypar}
Consequently, we derive by utilizing Lemma \ref{lemma:frechet_diff_bounded_domain} and the product rule for Fr{\'e}chet derivative that the partial derivative regarding the nonlocal operator can be expressed as
\begin{align*}
	\partial_t \varOpDer(t, \weakSol, \advar)& \\
	= \frac{1}{2} \iint_{(\completeDom)^2} &(\advar(\xb) - \advar(\yb))(\weakSol(\xb)\grad \Psi_1^t(\xb,\yb)^{\top}\Wb(\xb,\yb) - \weakSol(\yb)\grad \Psi_1^t(\yb,\xb)^{\top}\Wb(\yb,\xb))\xt(\xb)\xt(\yb) \\
	&+ (\advar(\xb) - \advar(\yb))(\weakSol(\xb) \Psi_1^t(\xb,\yb) - \weakSol(\yb)\Psi_1^t(\yb,\xb))\left. \frac{d}{dr} \right|_{r=t} (\xi^r(\xb) \xi^r(\yb)) ~d\yb d\xb\\
	+ \frac{1}{2} \iint_{(\completeDom)^2} &(\advar(\xb) - \advar(\yb))(\weakSol(\xb)\grad \Psi_2^t(\xb,\yb)^{\top}\Wb(\xb,\yb) - \weakSol(\yb)\grad \Psi_2^t(\yb,\xb)^{\top}\Wb(\yb,\xb))\xt(\xb)\xt(\yb) \\
	&+ (\advar(\xb) - \advar(\yb))(\weakSol(\xb)\Psi_2^t(\xb,\yb) - \weakSol(\yb)\Psi_2^t(\yb,\xb)) \left. \frac{d}{dr} \right|_{r=t} (\xi^r(\xb)\xi^r(\yb)) ~d\yb d\xb.
\end{align*}
\textbf{Singular symmetric kernels:}\\
As indicated in Remark \ref{remark:sing_kernel_bounded_Dn} Assumption (S1) yields $\Psi_{\shapetW,\Vb}^1,\Psi_{\shapetW,\Vb}^2 \in W^{1,\infty}(D_n)$ with weak derivatives 
\begin{align*}
	&\grad \Psi_{\shapetW,\Vb}^1(\xb,\yb)=\hess(\kernel_{\shapetW})(\xb,\yb)\Vb(\xb,\yb) + (\grad\kernel_{\shapetW}^{\top}(\xb,\yb) D\Vb(\xb,\yb))^{\top} \text{ and} \\
	&\grad \Psi_{\shapetW,\Vb}^2(\xb,\yb)=(\di \Vb(\xb) + \di \Vb(\yb))\grad\kernel_{\shapetW}(\xb,\yb) + \kernel_{\shapetW}(\xb,\yb)(\grad \di \Vb(\xb) + \grad \di \Vb(\yb)).
\end{align*}
 Then, by using Lemma \ref{lemma:frechet_diff_bounded_domain} and Assumption (S1) we get the differentiability of $\varOpDer$ regarding $t$ since
 \begin{align*}
 	\partial_t \varOpDer(t,\weakSol, \advar) = \lim_{n \rightarrow \infty} \frac{1}{2} \iint_{D_n} &(\advar(\xb) - \advar(\yb))(\weakSol(\xb) - \weakSol(\yb))\grad \Psi_1^t(\xb,\yb)^{\top}\Wb(\xb,\yb)\xt(\xb)\xt(\yb) \\
 	&+ (\advar(\xb) - \advar(\yb))(\weakSol(\xb) - \weakSol(\yb))\Psi_1^t(\xb,\yb) \left. \frac{d}{dr} \right|_{r=t} (\xi^r(\xb)\xi^r(\yb)) ~d\yb d\xb\\
 	+ \lim_{n \rightarrow \infty} \frac{1}{2} \iint_{D_n} &(\advar(\xb) - \advar(\yb))(\weakSol(\xb) - \weakSol(\yb))\grad \Psi_2^t(\xb,\yb)^{\top}\Wb(\xb,\yb)\xt(\xb)\xt(\yb)\\
 	&+(\advar(\xb) - \advar(\yb))(\weakSol(\xb) - \weakSol(\yb))\Psi_2^t(\xb,\yb)\left. \frac{d}{dr} \right|_{r=t} (\xi^r(\xb)\xi^r(\yb)) ~d\yb d\xb.
 \end{align*}
\end{proof}
\begin{lemma}\label{lemma:varOpDer_continuity}
	Let the Assumption (S1) be fulfilled.
	\begin{itemize}
		\item If $\kernel_{\shape}$ is an integrable kernel, then there exists a $C>0$ such that
		\begin{align*}
			|\varOpDer(t,\weakSol,\advar)| \leq C ||\weakSol||_{L^2(\completeDom)}||\advar||_{L^2(\completeDom)}. 
		\end{align*}
	\item If $\kernel_{\shape}$ is a singular symmetric kernel, then there exists a $C>0$ such that
	\begin{align*}
		|\varOpDer(t,\weakSol,\advar)| \leq C |\weakSol|_{H^s(\completeDom)}|\advar|_{H^s(\completeDom)}.
	\end{align*}
	\end{itemize}
\end{lemma}
\begin{proof}
	By utilizing the triangle inequality, we can directly deduce that
	\begin{align*}
		&|\varOpDer(t,\weakSol,\advar)| 
		\leq \sum_{i=1,2} |\iint_{(\completeDom)^2} (\advar(\xb) - \advar(\yb))(\weakSol(\xb)\Psi_i^t(\xb,\yb) - \weakSol(\yb)\Psi_i^t(\yb,\xb))\xt(\xb)\xt(\yb)~d\yb d\xb|.
	\end{align*}
\textbf{Integrable kernel:}\\
First, we denote that $\xt$ and $\di \Vbt$ are continuous on $\completeDom$ and therefore $$\Psi_2^t(\xb,\yb)\xt(\xb)\xt(\yb) =\kernelt(\xb,\yb)(\di \Vbt(\xb) + \di \Vbt(\yb))\xt(\xb)\xt(\yb)$$ is bounded by a constant $\widetilde{C}_1>0$ for all $t \in [0,T]$ due to Assumption (S1).
Then, we can conclude the existence of a $C_1>0$ such that for the term regarding $\Psi_2$ holds
\begin{align*}
	&|\iint_{(\completeDom)^2} (\advar(\xb) - \advar(\yb))(\weakSol(\xb)\Psi_2^t(\xb,\yb) - \weakSol(\yb)\Psi_2^t(\yb,\xb))\xt(\xb)\xt(\yb)~d\yb d\xb|\\
	&\leq \widetilde{C}_1 \iint_{(\completeDom)^2} |(\advar(\xb) - \advar(\yb))\weakSol(\xb)| + |(\advar(\xb) - \advar(\yb)) \weakSol(\yb)|~d\yb d\xb 
\leq C_1 ||\weakSol||_{L^2(\completeDom)}||\advar||_{L^2(\completeDom)}.
\end{align*}
Since also $\grad \kernel^t$ is essentially bounded for all $t \in [0,T]$, the term $$\Psi_1^t(\xb,\yb)\xt(\xb)\xt(\yb)=\grad \kernelt(\xb,\yb)^{\top}\Vbt(\xb,\yb)\xt(\xb)\xt(\yb)$$ is essentially bounded  on $(\completeDom)^2$ by a constant $\widetilde{C}_2>0$ for all $t \in [0,T]$. Consequently, we get for some $C_2 > 0$ that
\begin{align*}
	&|\iint_{(\completeDom)^2} (\advar(\xb) - \advar(\yb))(\weakSol(\xb)\Psi_1^t(\xb,\yb) - \weakSol(\yb)\Psi_1^t(\yb,\xb))\xt(\xb)\xt(\yb)~d\yb d\xb| \\
	&\leq \iint_{(\completeDom)^2} \widetilde{C}_2|(\advar(\xb) - \advar(\yb))\weakSol(\xb)| + \widetilde{C}_2|(\advar(\xb) - \advar(\yb))\weakSol(\yb)| ~d\yb d\xb 
	\leq C_2||\weakSol||_{L^2(\completeDom)} ||\advar||_{L^2(\completeDom)}.
\end{align*}
\textbf{Singular symmetric kernel:}\\
Here, we can directly conclude the existence of a constant $C_1>0$ with
\begin{align*}
	|\Psi_2^t(\xb,\yb)\xt(\xb)\xt(\yb)| \leq \frac{C_1}{||\xb-\yb||_2^{d+2s}} \ind_{B_{\delta}(\xb)}(\yb) =: \widetilde{\kernel}(\xb,\yb)\text{ for a. e. } (\xb,\yb) \in (\completeDom)^2
\end{align*}
and for all $t \in [0,T]$ due to Assumption (S1) and the essential boundedness of $\Vb$, $\di \Vb$ and $\xt$ on $\completeDom$. Therefore, we derive
\begin{align*}
	&|\iint_{(\completeDom)^2} (\advar(\xb) - \advar(\yb))(\weakSol(\xb) - \weakSol(\yb))\Psi_2^t(\xb,\yb)\xt(\xb)\xt(\yb) ~d\yb d\xb|\\
	&\leq \iint_{(\completeDom)^2} |(\advar(\xb) - \advar(\yb))(\weakSol(\xb) - \weakSol(\yb))| \widetilde{\kernel}(\xb,\yb) ~d\yb d\xb \\
	&\leq \left( \iint_{(\completeDom)^2} (\advar(\xb)-\advar(\yb))^2\widetilde{\kernel}(\xb,\yb) ~d\yb d\xb \right)^{\frac{1}{2}} \left(\iint_{(\completeDom)^2} (\weakSol(\xb)-\weakSol(\yb))^2\widetilde{\kernel}(\xb,\yb)~d\yb d\xb \right)^{\frac{1}{2}} \\
	&\leq C_1 |\weakSol|_{H^s(\completeDom)} |\advar|_{H^s(\completeDom)},
\end{align*}
\begin{sloppypar}
where we obtain the last step by dropping $\ind_{B_{\delta}(\xb)}$.
From Assumption (S1) follows the existence of a constant $C_2>0$ with ${|\Psi_1^t(\xb,\yb)\xt(\xb)\xt(\yb)|\leq \frac{C_2}{||\xb-\yb||_2^{d+2s}}\ind_{B_{\delta}(\xb)}(\yb) =: \overline{\kernel}(\xb,\yb)}$ for a.e. ${(\xb,\yb) \in (\completeDom)^2}$ and for all $t \in [0,T]$. Accordingly, we derive
\end{sloppypar}
\begin{align*}
	&|\iint_{(\completeDom)^2} (\advar(\xb) - \advar(\yb))(\weakSol(\xb) - \weakSol(\yb))\Psi_1^t(\xb,\yb)\xt(\xb)\xt(\yb)~d\yb d\xb| \\
	&\leq \left(\iint_{(\completeDom)^2} (\advar(\xb) - \advar(\yb))^2 \overline{\kernel}(\xb,\yb) ~d\yb d\xb\right)^{\frac{1}{2}} \left( \iint_{(\completeDom)^2} (\weakSol(\xb) - \weakSol(\yb))^2\overline{\kernel}(\xb,\yb) ~d\yb d\xb\right)^{\frac{1}{2}}\\
	&\leq C_2 |\weakSol|_{H^s(\completeDom)} |\advar|_{H^s(\completeDom)},
\end{align*}
where we again omitted the truncation $\ind_{B_{\delta}(\xb)}$ in the last step.
\end{proof}
\begin{lemma}[After {\cite[Chapter 10.2.4 Lemma 2.1]{shapes_geometries}}]
	\label{lemma:l2_convergence}
	Let $\holdAll \subset \Rd$ be a bounded and open domain with nonzero measure, $\Vb \in C_0^1(\holdAll,\Rd)$ and $n \in \Nbb$. Then, for $g \in L^2(\holdAll,\R^n)$, we get
	\begin{align*}
		||g \circ \FbtV - g||_{L^2(\holdAll,\R^n)} \rightarrow 0 \text{ and } ||g \circ (\FbtV)^{-1} - g||_{L^2(\holdAll,\R^n)} \rightarrow 0 \quad \text{for } t \searrow 0.
	\end{align*}
\end{lemma} 
\begin{proof}
\begin{sloppypar}
Since $g$ and $\Vb$ can be extended by zero to functions in $L^2(\Rd,\R^n)$ and ${C_c^1(\Rd,\R^n)\defas\{\Vb \in C_c^1(\Rd,\R^n): \supp (\Vb) \text{ compact in } \Rd \}}$, respectively, we refer for the case $n=1$ to the proof of \cite[Chapter 10.2.4 Lemma 2.1]{shapes_geometries}. Then, the cases $n \in \Nbb$ with $n \geq 2$ are a direct consequence.
\end{sloppypar} 
\end{proof}
Before we prove that the conditions of (D3) are met under some natural assumptions, we need the following weak convergence results. 
\begin{lemma}\label{lemma:weak_conv} 
Suppose that Assumption (S0) and (S1) hold. Then, for $t \searrow 0$ we get
	\begin{align*}
		\weakSol^t \rightharpoonup \weakSol^0, \advar^t \rightharpoonup \advar^0, \psi^t \rightharpoonup \psi^0 \text{ and } \varphi^t \rightharpoonup \varphi^0.
	\end{align*} 
\end{lemma}
\begin{proof}
	Due to its length, the proof has been moved to Appendix \ref{app:weak_conv_proof}.
\end{proof}
\begin{lemma}\label{lemma:D3_holds}
	If Assumptions (S0) and (S1) are satisfied, Assumption (D3) of the AAM holds.
\end{lemma}
\begin{proof}
	Since the proof is quite lengthy, it can be found in Appendix \ref{app:D3_proof}.
\end{proof}
Note that $(\grad \di \Vb)^{\top}\Wb = \di (D\Vb \Wb) - \trace(D\Vb D\Wb)$ (see, e.g., \cite{linear_view}). Additionally, we decompose the terms $\grad (\Psi_{\shape,\Vb}^i(\xb,\yb))^{\top}\Wb(\xb,\yb)$ in two parts each, such that we can express $D_{\shape}^2\redFun(\shape)[\Vb,\Wb]$ according to \eqref{eq:structure_second_shape}. Therefore, we set
\begin{align}
	\begin{split}\label{def:grad_decomposition}
	T_1^1(\xb,\yb) =& \Vb(\xb,\yb)^{\top}\hess(\kernel_{\shape})(\xb,\yb)\Wb(\xb,\yb),\quad  T_2^1(\xb,\yb) = \grad\kernel_{\shape}^{\top}(\xb,\yb) D\Vb(\xb,\yb)\Wb(\xb,\yb), \\
	T_1^2(\xb,\yb) =& (\di \Vb(\xb) + \di \Vb(\yb))\grad\kernel_{\shape}(\xb,\yb)^{\top}\Wb(\xb,\yb) \\
	&- \kernel_{\shape}(\xb,\yb)(\trace(D\Vb(\xb)D\Wb(\xb)) + \trace (D\Vb(\yb)D\Wb(\yb))) \text{ and} \\
	T_2^2(\xb,\yb) =& \kernel_{\shape}(\xb,\yb)(\di(D\Vb(\xb)\Wb(\xb)) + \di (D\Vb(\yb)\Wb(\yb))).
	\end{split}
\end{align}
Consequently, we derive
\begin{align*}
	(\grad \Psi_{\shape,\Vb}^1(\xb,\yb))^{\top}\Wb(\xb,\yb) = T_1^1(\xb,\yb) + T_2^1(\xb,\yb) \text{ and } (\grad \Psi_{\shape,\Vb}^2(\xb,\yb))^{\top}\Wb(\xb,\yb) = T_1^2(\xb,\yb) + T_2^2(\xb,\yb).
\end{align*}
\begin{corollary}
	Let the functions $\weakSol^0,\advar^0$ solve \eqref{eq:state_1} and \eqref{eq:state_2}. Further, suppose that the functions $\psi^0,\varphi^0$ are the solutions to \eqref{eq:AAE1_nonlocal} and \eqref{eq:AAE2_nonlocal}. Additionally, the functions $\Psi_{\shape,\Vb}^1$ and $\Psi_{\shape,\Vb}^2$ are defined as in \eqref{defs:Phis} and their gradient terms can be written as indicated in \eqref{def:grad_decomposition}.
	Then, the second shape derivative of the reduced functional can be expressed as
	\begin{align*}
		&D^2_{\shape} \redFun(\shape)[\Vb,\Wb] (= \partial_t \reallagrangian(0, \weakSol^0,\advar^0,\psi^0,\varphi^0)) = \objFun''(\shape)[\Vb,\Wb] + D_{\shape} \objFun(\shape)[D\Vb\Wb],
	\end{align*}
where the first component is the linear second shape derivative
	\begin{align*}
		&\objFun''(\shape)[\Vb,\Wb] \\
		&= \int_{\nlDom}  \grad\data^{\top}\Wb \grad\data^{\top}\Vb - (\weakSol^0 - \data)\left(\Vb^{\top} \hess(\data)\Wb + \grad\data^{\top}\Vb \di \Wb \right) ~d\xb\\
		&+ \int_{\nlDom} -(\weakSol^0 - \data)\grad\data^{\top}\Wb \di \Vb + \frac{1}{2} \left(\weakSol^0 - \data\right)^2 \left(-\trace(D\Vb D\Wb) + \di \Vb \di \Wb\right) ~d\xb \\
		&+\int_{\nlDom} \Vb^{\top}\hess(f_{\shape})\Wb \advar^0 + \grad f_{\shape}^{\top}\Vb\advar^0 \di \Wb~d\xb \\
		&+ \int_{\nlDom} \grad f_{\shape}^{\top} \Wb \advar^0 \di \Vb + f_{\shape} \advar^0 \left(-\trace(D\Vb D\Wb) + \di \Vb \di \Wb \right) ~d\xb \\
		&+\sum_{i=1}^2 \frac{1}{2} \Biggl( \iint_{(\completeDom)^2} (\advar^0(\xb) - \advar^0(\yb))(\weakSol^0(\xb) T_1^i(\xb,\yb) - \weakSol^0(\yb)T_1^i(\yb,\xb)) ~d\yb d\xb \\
		&+ \iint\limits_{(\completeDom)^2} (\advar^0(\xb) - \advar^0(\yb))(\weakSol^0(\xb)\Psi_{\shape,\Vb}^i(\xb,\yb) - \weakSol^0(\yb)\Psi_{\shape,\Vb}^i(\yb,\xb))(\di \Wb(\xb) + \di \Wb(\yb)) ~d\yb d\xb \Biggr)\allowdisplaybreaks\\
		&+ \sum_{i=1}^2 \frac{1}{2} \left(\iint_{(\completeDom)^2} (\psi^0(\xb) - \psi^0(\yb))(\weakSol^0(\xb)\Psi_{\shape,\Wb}^i(\xb,\yb) - \weakSol^0(\yb)\Psi_{\shape,\Wb}^i(\yb,\xb)) ~d\yb d\xb\right) \\
		&+ \sum_{i=1}^2 \frac{1}{2} \left(\iint_{(\completeDom)^2} (\advar^0(\xb) - \advar^0(\yb))(\varphi^0(\xb)\Psi_{\shape,\Wb}^i(\xb,\yb) - \varphi^0(\yb)\Psi_{\shape,\Wb}^i(\yb,\xb)) ~d\yb d\xb \right) \\
		&-\int_{\nlDom} \grad f^{\top}\Wb\psi^0 + f \psi^0 \di \Wb ~d\xb + \int_{\nlDom} (\weakSol^0 - \data)\varphi^0\di \Wb - \grad \data^{\top}\Wb \varphi^0 ~d\xb.
	\end{align*}
and where the second term is the shape derivative at $\shape$ in direction $D\Vb\Wb$, which can be written as
\begin{align*}
D_{\shape} \objFun(\shape)[D\Vb\Wb] = &\int_{\nlDom} - \left(\weakSol^0 - \data\right)\grad \data^{\top}D\Vb\Wb + \frac{1}{2}(\weakSol^0 - \data)^2 \di(D\Vb\Wb) ~d\xb \\
&+ \int_{\nlDom} \grad f_{\shape}^{\top}D\Vb\Wb\advar^0 + f_{\shape}\advar^0 \di(D\Vb\Wb) ~d\xb \\
&+ \sum_{i=1}^2 \frac{1}{2} \iint_{(\completeDom)^2} (\advar^0(\xb) - \advar^0(\yb)) (\weakSol^0(\xb)T_2^i(\xb,\yb) - \weakSol^0(\yb)T_2^i(\yb,\xb)) ~d\yb d\xb.	
\end{align*}
\end{corollary}
We now continue by putting the derived formula for $(\redFun)''(\shape)[\Vb,\Wb]$ into numerical practice.
\section{Numerical Experiments}
\label{chap:num_exp}
Here, we present a second order shape optimization algorithm, where we follow a Newton-like approach, that incorporates first and second shape derivatives of the reduced objective functional, which are derived by the averaged adjoint method. Additionally, we describe the implementation and discuss numerical examples.\\~\\
As mentioned before, we follow \cite{linear_view} and optimize on a suitable subset of $\{T|T:\overline{\nlDom} \rightarrow \overline{\nlDom}\}$. Therefore, we omit the possibly nonsymmetric part of the second shape derivative and only utilize $(\objFun^{red})''(\shape)[\Vb,\Wb]$. Here, the objective functional also includes the perimeter regularization $\perimeter$, whose shape derivatives can be found in Remark \ref{remark:perimeter_der}.
Moreover, we apply the finite element method and approximate functions by a linear combination of continuous, piecewise linear basis functions. In particular, in every iteration $k$ we will search for a deformation $\Wb^{k,h}$ in the space of such linear combinations of two dimensional CG ansatz functions, which we denote as $CG^2$.
For a detailed description on how the finite element framework can be implemented in order to solve nonlocal Dirichlet problems we refer to \cite{fe_cookbook}.\\~\\
However, due to the Hadamard structure theorem (see \cite[Chapter 9 Theorem 3.6 and Corollary 1]{shapes_geometries}), the support of every Riesz representation of $D_{\shape} \objFun^{red}(\shape)[\Vb]$ is a subset of $\shape$ and is only dependent on the normal component of a vector field $\Vb$ on $\shape$, if $\shape$ is smooth enough. Therefore, deformations inside of $\nlDom_1$ and $\nlDom_2$ as well as tangential forces do not result in a change of the objective functional value. Then, from Definition \ref{def:second_shape_der} and the structure of the second shape derivative (see \eqref{eq:structure_second_shape}) follows that
the linear second shape derivative $(\objFun^{red})''(\shape)[\Vb,\Wb]$ has a huge null space, is not invertible and thus the corresponding stiffness matrix in the finite element method is not positive definite. As a remedy, we set
\begin{align*}
	\reg(\Vb,\Wb) \defas \int_{\nlDom} \left(\Vb,\Wb\right)_2 + \left(\grad \Vb, \grad \Wb\right)_F ~d\xb
\end{align*}
and apply a Newton-like approach by solving the following problem in every iteration for a given constant $\epsilon > 0$.
\begin{align}
	\begin{split} \label{eq:Newton_like_method}
		&\textit{Find } \Wb_{\epsilon}^k \in \vecfieldsspecific \textit{ subject to } \\
		&(\objFun^{red})''(\shape^k)[\Vb, \Wb^k_{\epsilon}] + \epsilon \reg(\Vb, \Wb_{\epsilon}^k) = - D_{\shape} \objFun^{red}(\shape^k)[\Vb] \textit{ for all } \Vb \in \vecfieldsspecific.
	\end{split}
\end{align}
Then, there exists a unique solution to problem \eqref{eq:Newton_like_method} and the finite element stiffness matrix is positive definite.
As shown in \cite[Theorem 3.5]{linear_view}, if $\epsilon \searrow 0$ the corresponding solutions $\Wb^k_{\epsilon}$ converge to a solution that is derived by solving the unperturbed problem \eqref{eq:Newton_like_method} with $\epsilon = 0$ by a Newton Method based on the Moore-Penrose pseudoinverse.\\
Recall that in order to compute $D_{\shape} \objFun^{red}(\shape^k)[\Vb]$ and $(\objFun^{red})''(\shape^k)[\Vb,\Wb]$, we need to solve \eqref{eq:state_1} and \eqref{eq:state_2}. Additionally, for every $\Vb \in \vecfieldsspecific$ with $\supp(\Vb) \cap \shape^k \neq \emptyset$, we have to compute solutions $\psi(\shape^k,\Vb)$ and $\varphi(\shape^k, \Vb)$ to \eqref{eq:AAE1_nonlocal} and \eqref{eq:AAE2_nonlocal}.
The complete optimization procedure is illustrated in Algorithm \ref{algorithm}.\\
\begin{algorithm}[H]\label{algorithm}
	\While{$k \leq$ \texttt{maxiter} $ \text{ or } \| \Wb^{k,h}_{\epsilon} \| > \text{ tol} $}{
		Interpolate $\data$ onto current finite element mesh $\nlDom^k$. \\
		$\weakSol(\shape^k),\advar(\shape^k) \leftarrow$ Solve state and adjoint equation.\\
		Assemble shape derivative $D_{\shape} \objFun^{red}(\shape^k)[\Vb^h]$, for all $\Vb^h \in CG^2$.\\
		Set $D_{\shape} \objFun^{red}(\shape^k)[\Vb^h] = 0$ if $\supp(\Vb^h)\cap \shape^k=\emptyset$.\\
		$\psi(\shape^k, \Vb^h), \varphi(\shape^k, \Vb^h) \leftarrow$ Solve \eqref{eq:AAE1_nonlocal} and \eqref{eq:AAE2_nonlocal} for all $\Vb^h \in CG^2$.\\
		Assemble linear second shape derivative $(\objFun^{reg})''(\shape^k)[\Vb^h, \Wb^h]$, for all $\Vb^h, \Wb^h \in CG^2$.\\
		$\Wb^{k,h}_{\epsilon} \leftarrow$ Get deformation by solving 
		\begin{align*}
		(\objFun^{red})''(\shape^k)[\Vb^h,\Wb^{k,h}_{\epsilon}] + \epsilon \reg(\Vb^h,\Wb^{k,h}_{\epsilon}) = - D\objFun^{red}(\shape^k)[\Vb^h] \text{ for all } \Vb^h \in CG^2.
		\end{align*}
		Update mesh:\\
		$\nlDom^{k+1}= (\Id + \Wb^{k,h}_{\epsilon})(\nlDom^k)$\\
		$k = k+1$
	} \caption{Second Order Shape Optimization Algorithm}
\end{algorithm}~\\
As mentioned before, we apply the finite element method. For our experiments, the underlying mesh is generated by Gmsh\cite{gmsh} and we assemble the stiffness matrix for the bilinear nonlocal operator $\varOp$ with the Python package nlfem\cite{nlfem}. Moreover, the shape derivatives as well as the linear second derivative of the nonlocal bilinear operator $\varOp_{\shape}$ are built by a customized version of nlfem. The regularization function and the objective functional along with their derivative are constructed by FEniCS\cite{FEniCS1,FEniCS2}. \\
For the first example, we choose an integrable kernel $\kernel_1$ as follows
\begin{align*}
	\kernel_1(\xb,\yb) = \begin{cases}
		0.1c_1\ind_{B_{\delta}(\xb)}(\yb) & \text{if } \xb,\yb \in \nlDom_1, \\
		10c_1\ind_{B_{\delta}(\xb)}(\yb) & \text{if } \xb,\yb \in  \nlDom_2, \\
		c_1\ind_{B_{\delta}(\xb)}(\yb) & \text{else},
	\end{cases}
\end{align*}
where $c_1 \defas \frac{1}{\delta^4}$. Additionally, we set $\delta =0.1$, $\epsilon=0.3$, $\nu=2\cdot10^{-3}$, $g=0$ as boundary data and ${f_{\shape}= 10 \ind_{\nlDom_1} - 10 \ind_{\nlDom_2}}$ in the forcing term.
\begin{figure}[h!] 
	\begin{center}
		\begin{small}
			\begin{tabular}{cccc}
				\includegraphics[width = 0.2\textwidth]{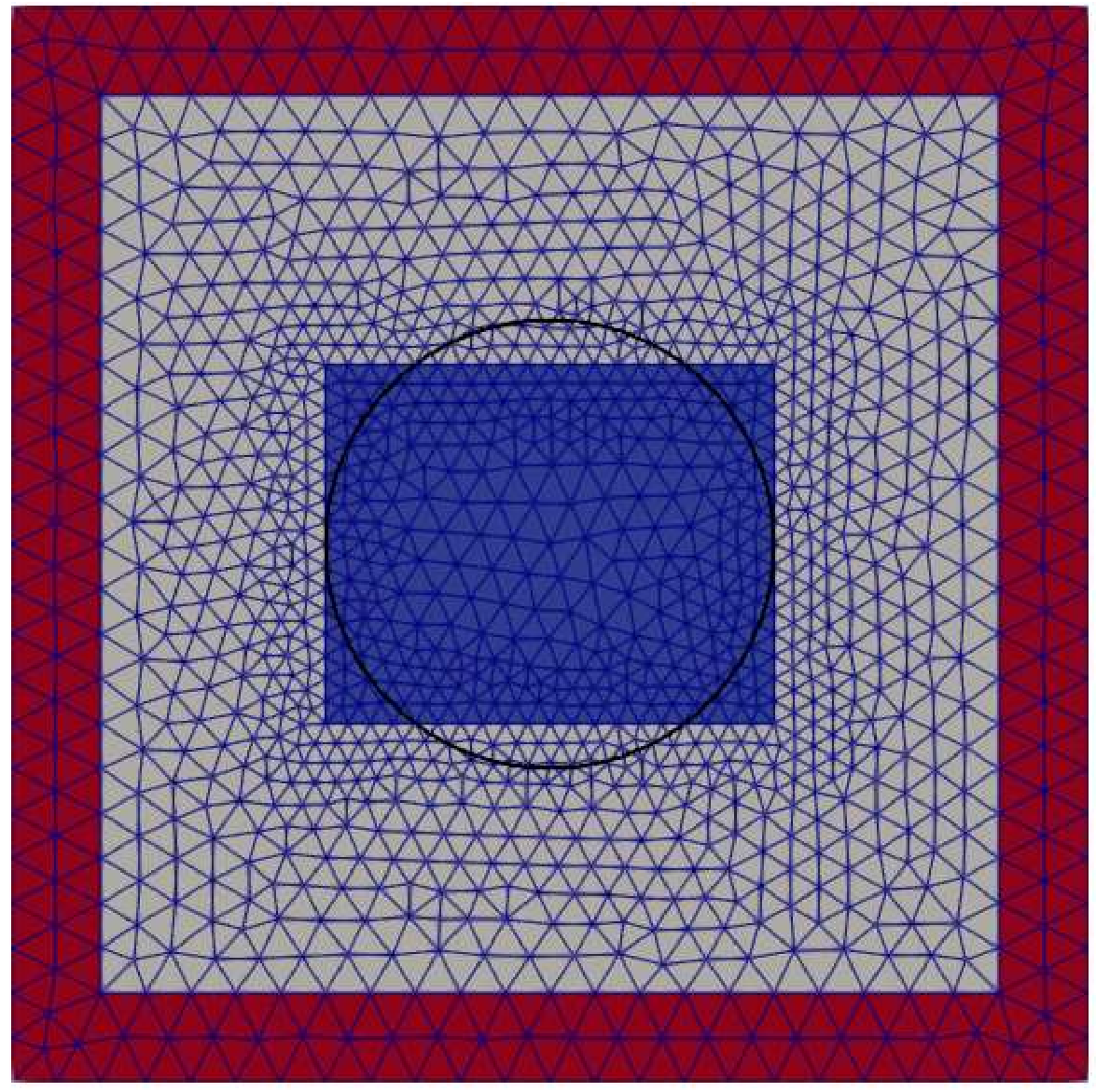}
				&\includegraphics[width = 0.2\textwidth]{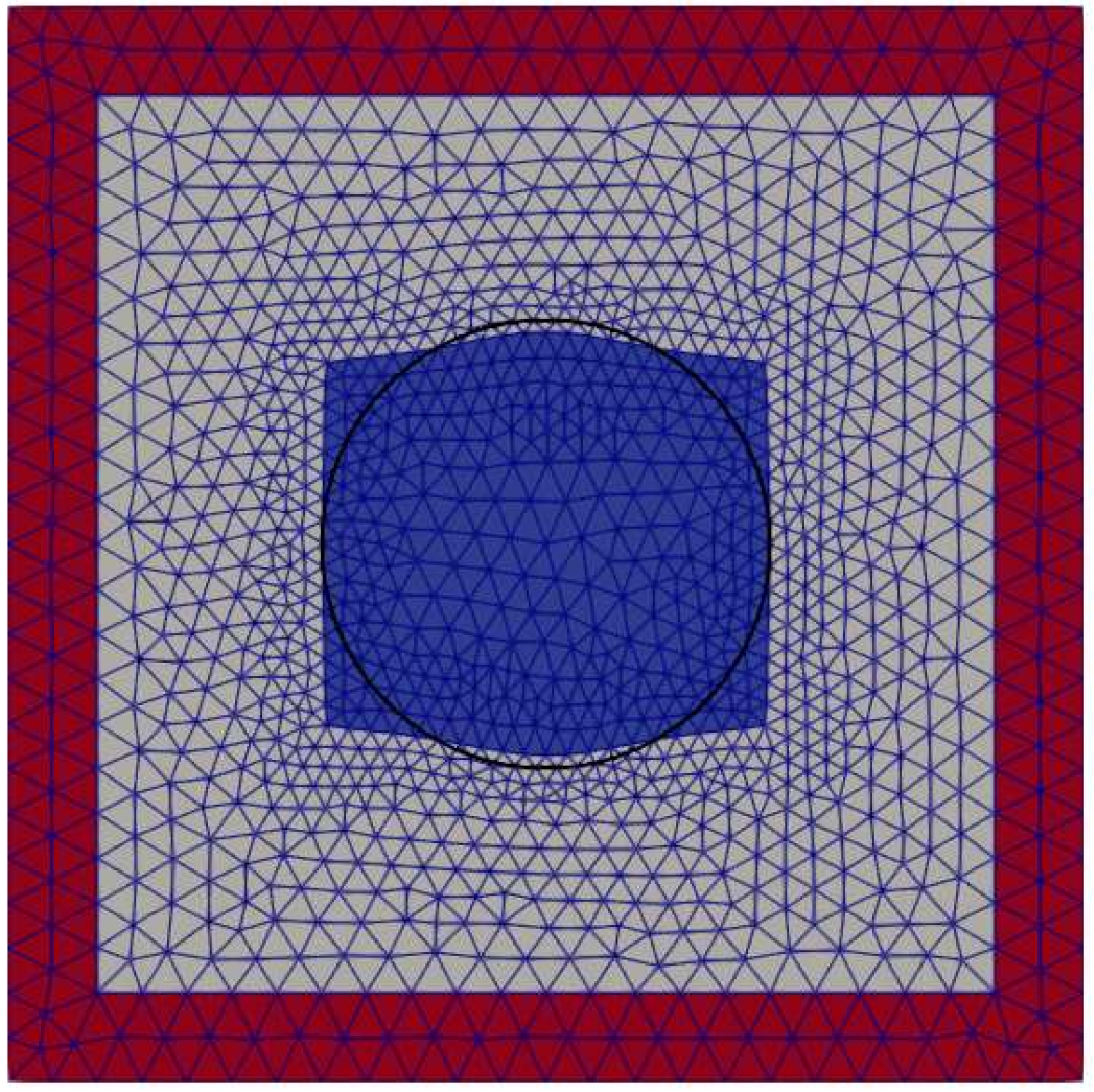}
				&\includegraphics[width = 0.2\textwidth]{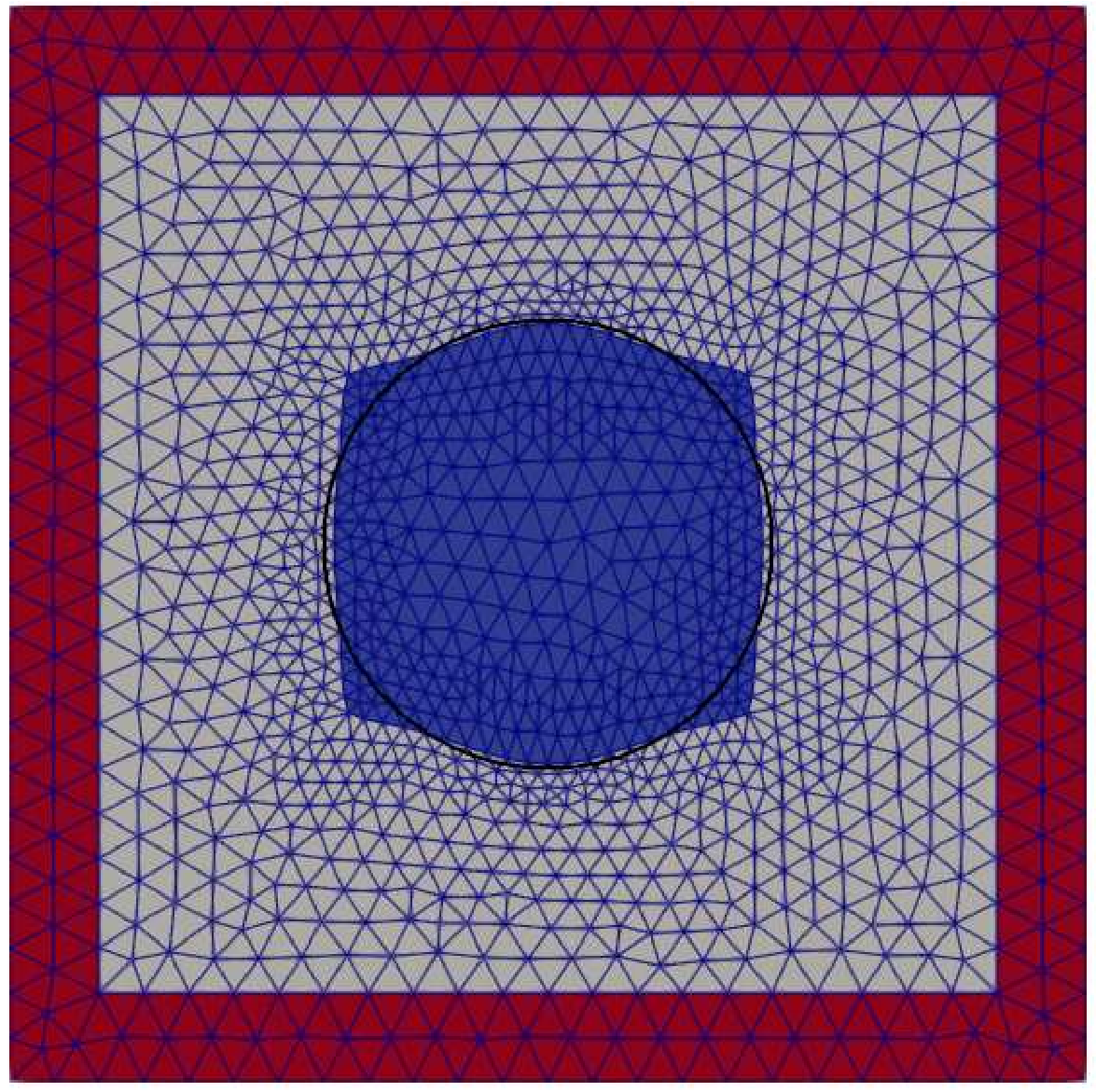}
				&\includegraphics[width = 0.2\textwidth]{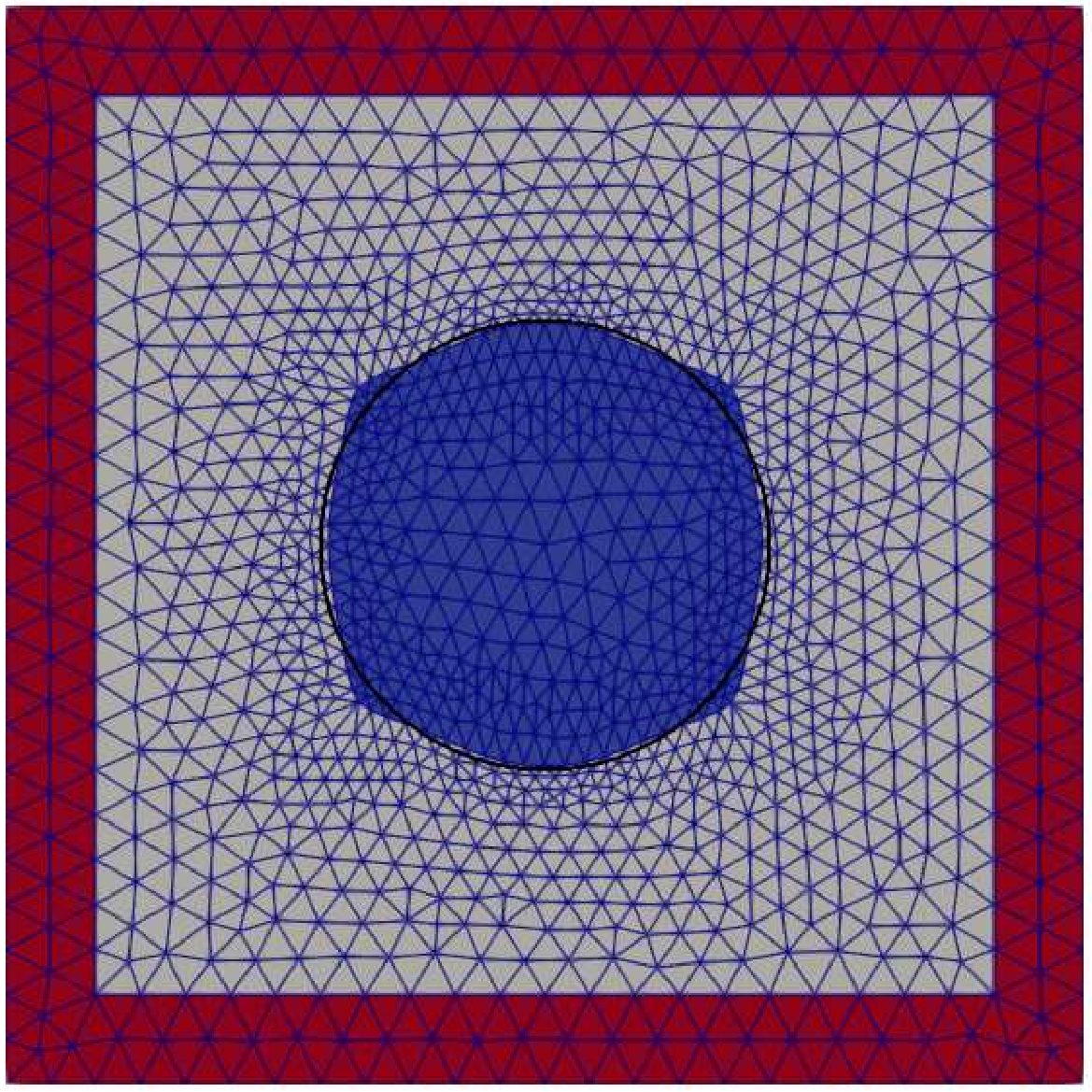}\\
				Start setup & Iteration 1 & Iteration 5 & Iteration 10\\
				&\includegraphics[width = 0.2\textwidth]{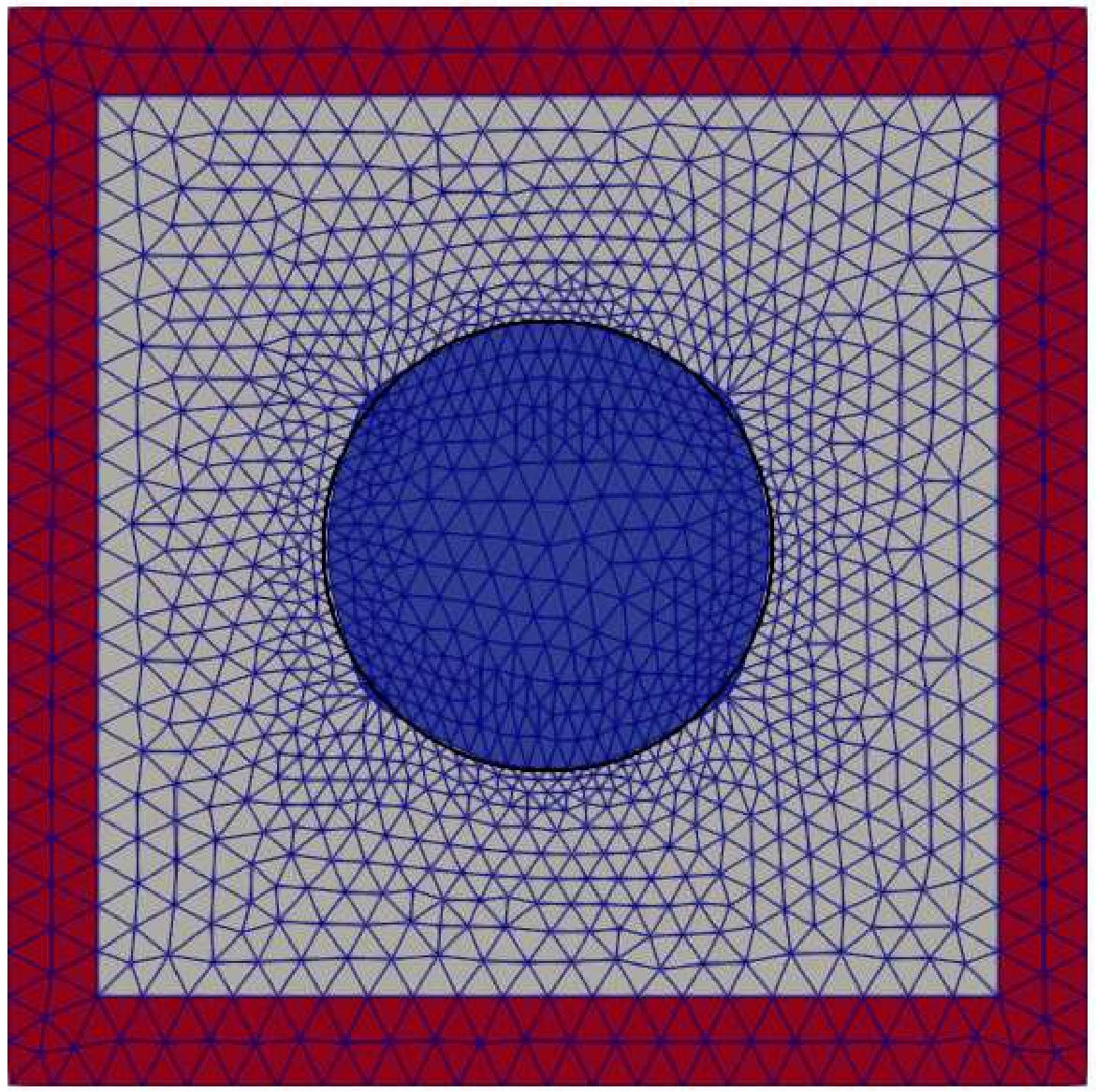}
				&\includegraphics[width = 0.2\textwidth]{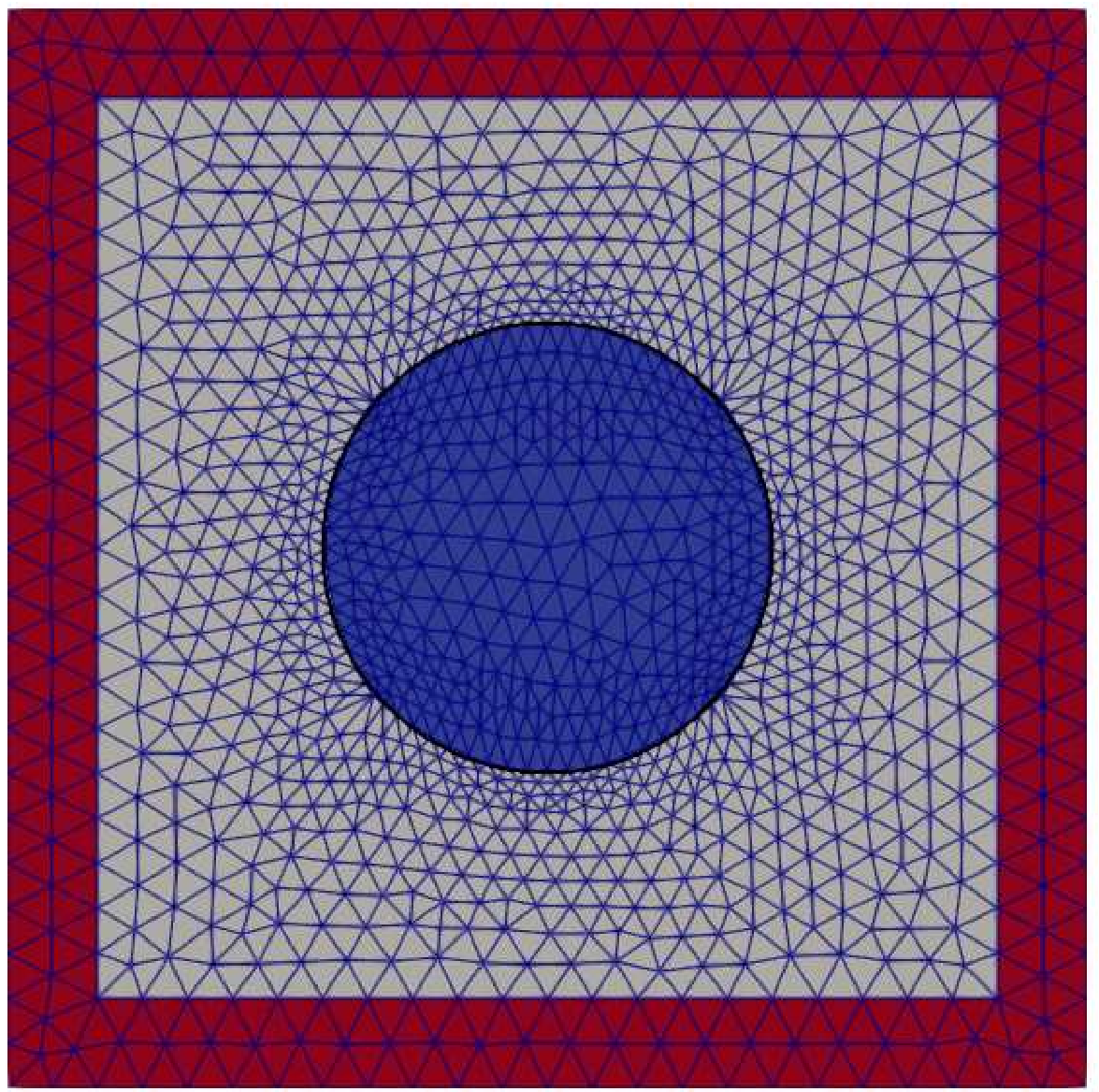}
				&\includegraphics[width = 0.2\textwidth]{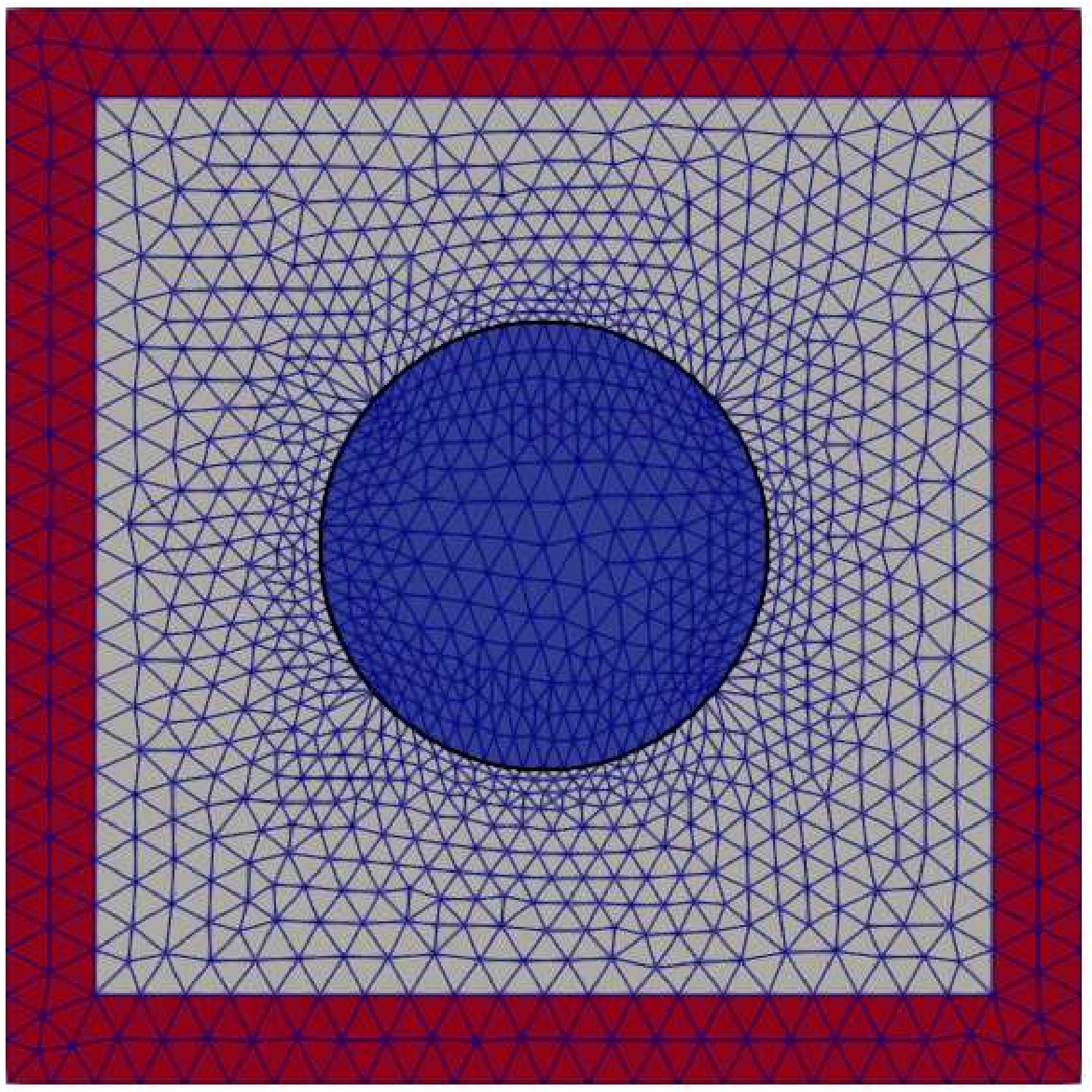}\\
				& Iteration 15 & Iteration 20 & Iteration 29
			\end{tabular}
		\end{small}
	\end{center}
	\caption{Example 1: The domain $\nlDom$ is decomposed in the blue area $\nlDom_1$ and the gray region $\nlDom_2$. The nonlocal boundary is colored in red. Additionally, the black circle can be interpreted as the target shape, i.e., we solve the nonlocal Dirichlet problem given the integrable kernel $\kernel_1$, the parameters chosen as described in the text and the decomposition of $\nlDom$ in $\nlDom_1$ and $\nlDom_2$ according to the black circle. The solution is then used as the data $\data$ in the interface identification problem, where we start with the blue square as the initial shape of $\nlDom_1$.}
	\label{fig:shape_ex_1}
\end{figure}
The set-up and the results are shown in Figure \ref{fig:shape_ex_1}. In our experiments, the data $\data$, which we would like to approximate as good as possible, is computed by solving the nonlocal Dirichlet problem, where $\shape$ is the black circle, that can be seen in all the picture, i.e., $\nlDom_1$ is inside and $\nlDom_2$ outside of $\shape$. Then, we start in a set-up, where $\nlDom_1$ is a square (see Figure \ref{fig:shape_ex_1}) and, ideally, this blue area should be deformed to a blue disk in the course of the shape optimization algorithm. We stop, after 50 iterations or if the $L^2(\nlDom)$-norm of the deformation $\Wb^{k,h}_{\epsilon}$ is smaller than $5\cdot10^{-5}$, which we again compute by using FEniCS. As we can observe in Figure \ref{fig:shape_ex_1}, we derive roughly a circle after 15 iterations and the algorithm terminates after 29 iterations.\\
In the second example, we apply the singular symmetric kernel
\begin{align*}
	\kernel_2(\xb,\yb) = \begin{cases}
		10.0 \frac{c_2}{||\xb-\yb||_2^{2+2s}} \ind_{B_{\delta}(\xb)}(\yb) & \text{if } \xb,\yb \in \nlDom_1 \\
		1.0 \frac{c_2}{||\xb-\yb||_2^{2+2s}} \ind_{B_{\delta}(\xb)}(\yb) & \text{if } \xb,\yb \in \nlDom_2 \\
		5.0 \frac{c_2}{||\xb-\yb||_2^{2+2s}} \ind_{B_{\delta}(\xb)}(\yb) & \text{else},
	\end{cases}
\end{align*}
where $c_2 \defas \frac{2 - 2s}{\pi \delta^{2+2s}}$. Further, we set $\epsilon = 0.02$, $\nu = 2\cdot10^{-4}$ and $\delta=0.1$. Again we have $f_{\shape} = 10\ind_{\nlDom_1} - 10\ind_{\nlDom_2}$ in the forcing term and $g = 0$ on the boundary. The termination criteria as well as the target shape are also the same as in the first experiment. The results are shown in Figure \ref{fig:shape_ex_2}, where we now start from a different shape $\shape$. In this case, $\nlDom_1$ has roughly the shape of the black circle after 5 iterations and the algorithm ends after 22 iterations.
\begin{figure}[h!] 
	\begin{center}
		\begin{small}
			\begin{tabular}{cccc}
				\includegraphics[width = 0.2\textwidth]{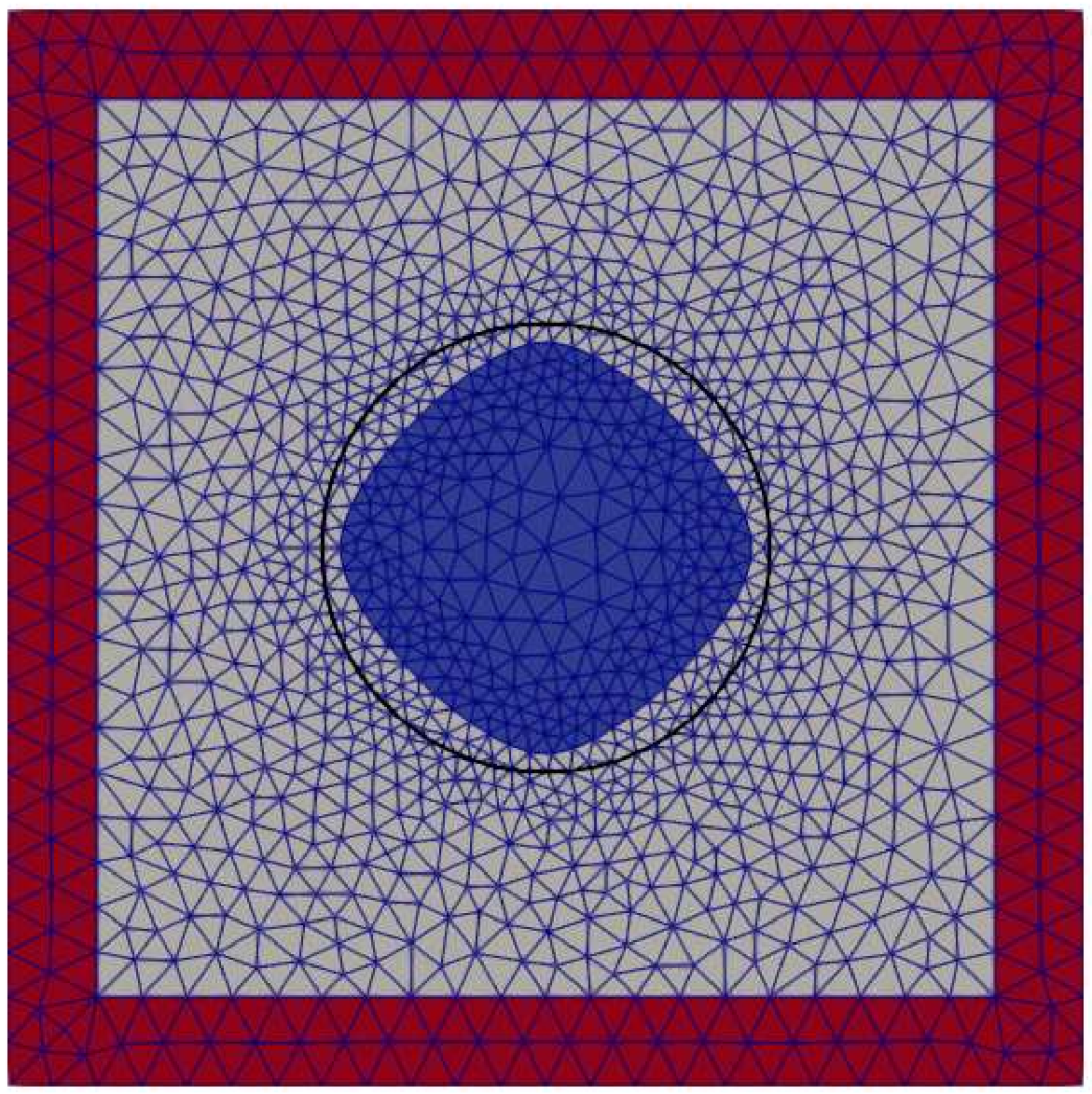}
				&\includegraphics[width = 0.2\textwidth]{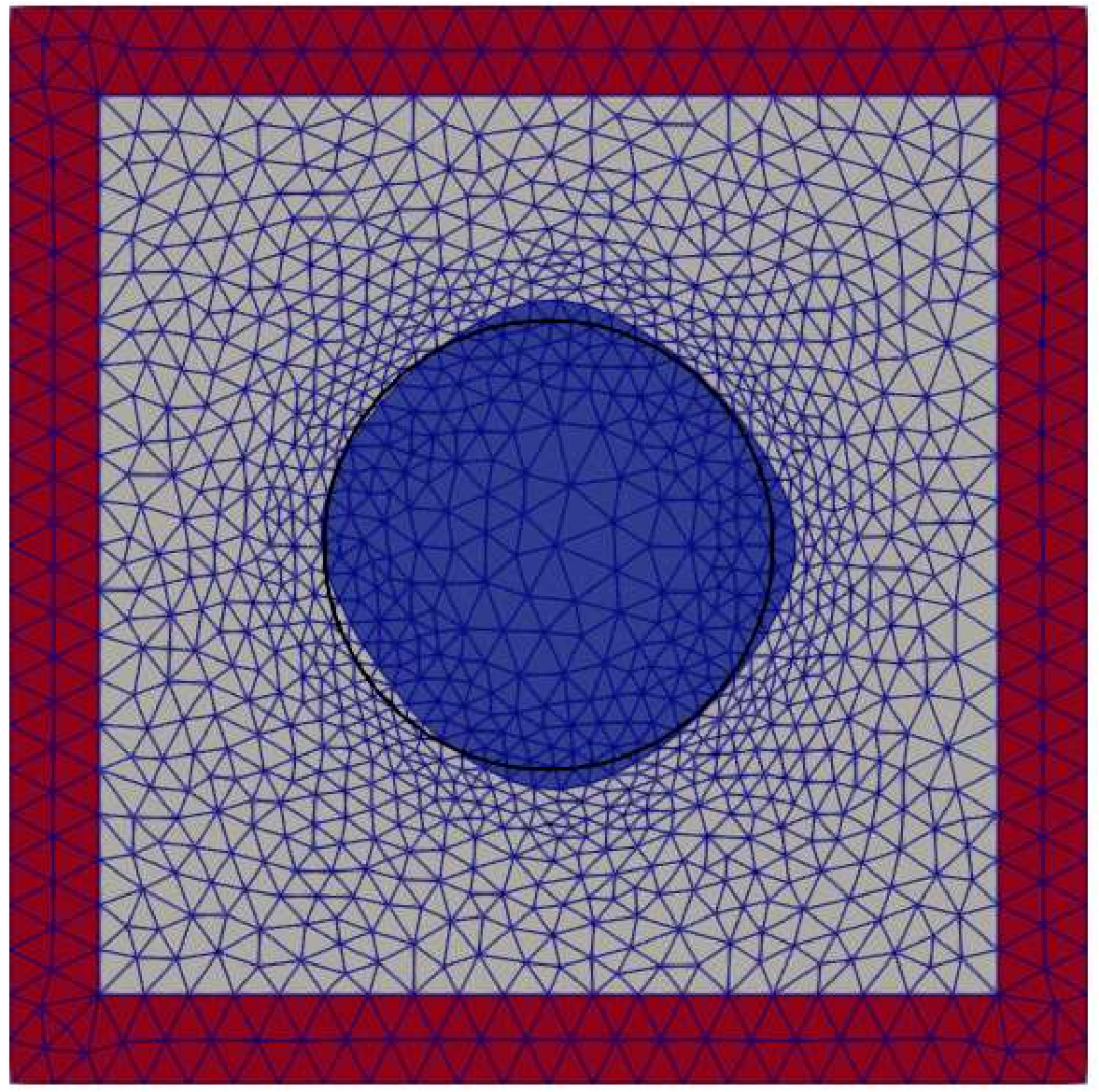}
				&\includegraphics[width = 0.2\textwidth]{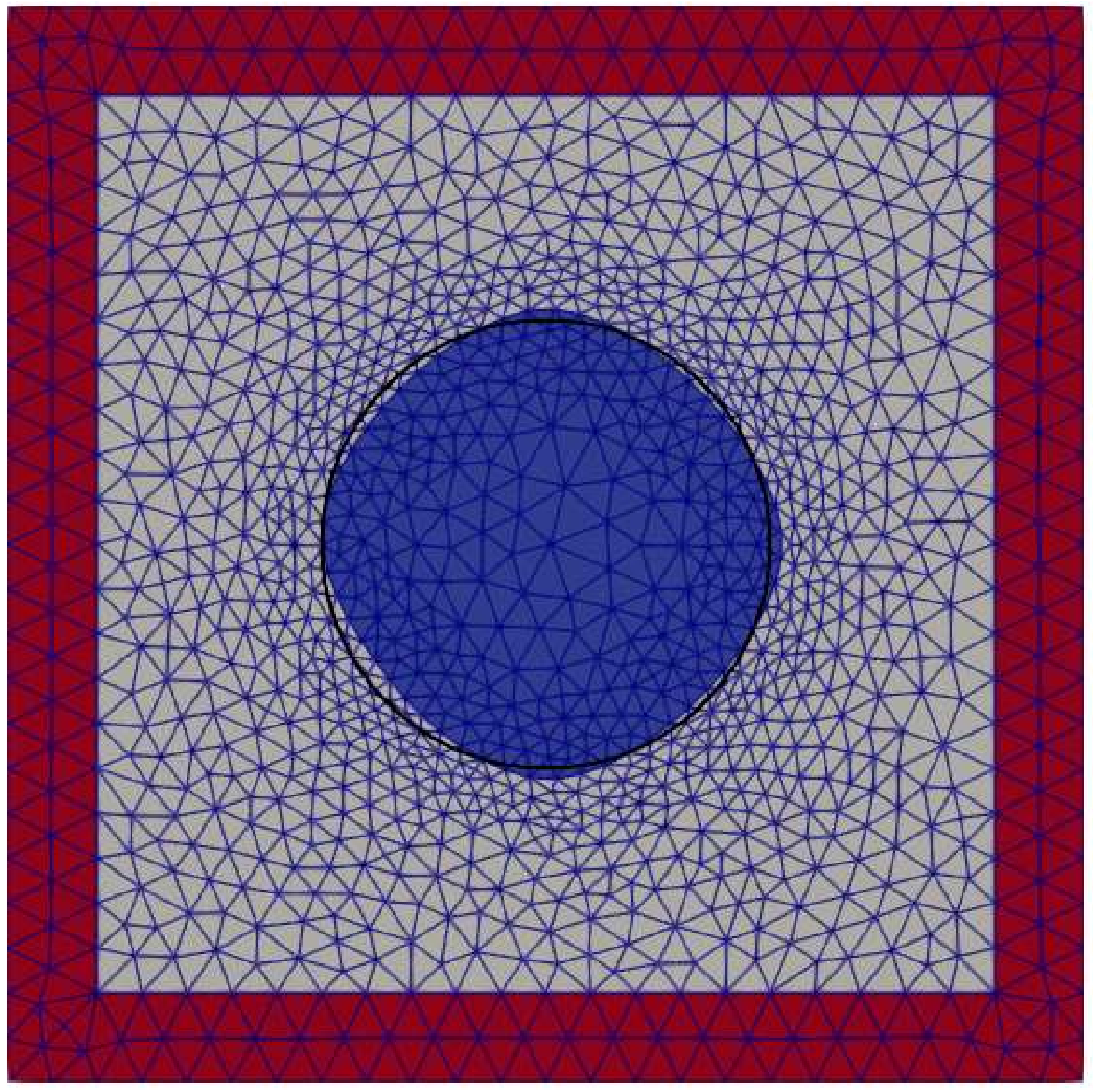}
				&\includegraphics[width = 0.2\textwidth]{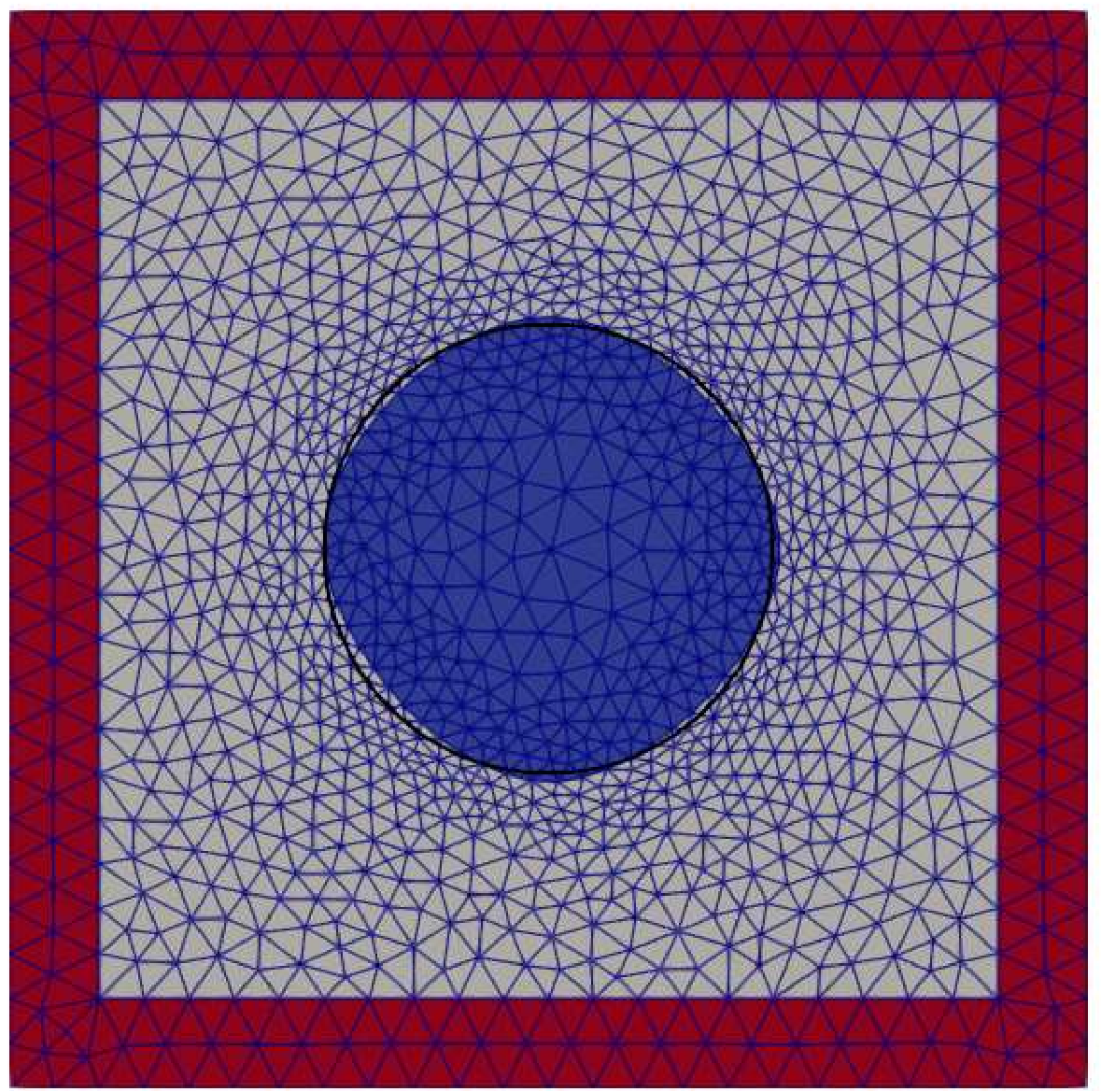}\\
				Start setup & Iteration 1 & Iteration 2 & Iteration 5\\
				&\includegraphics[width = 0.2\textwidth]{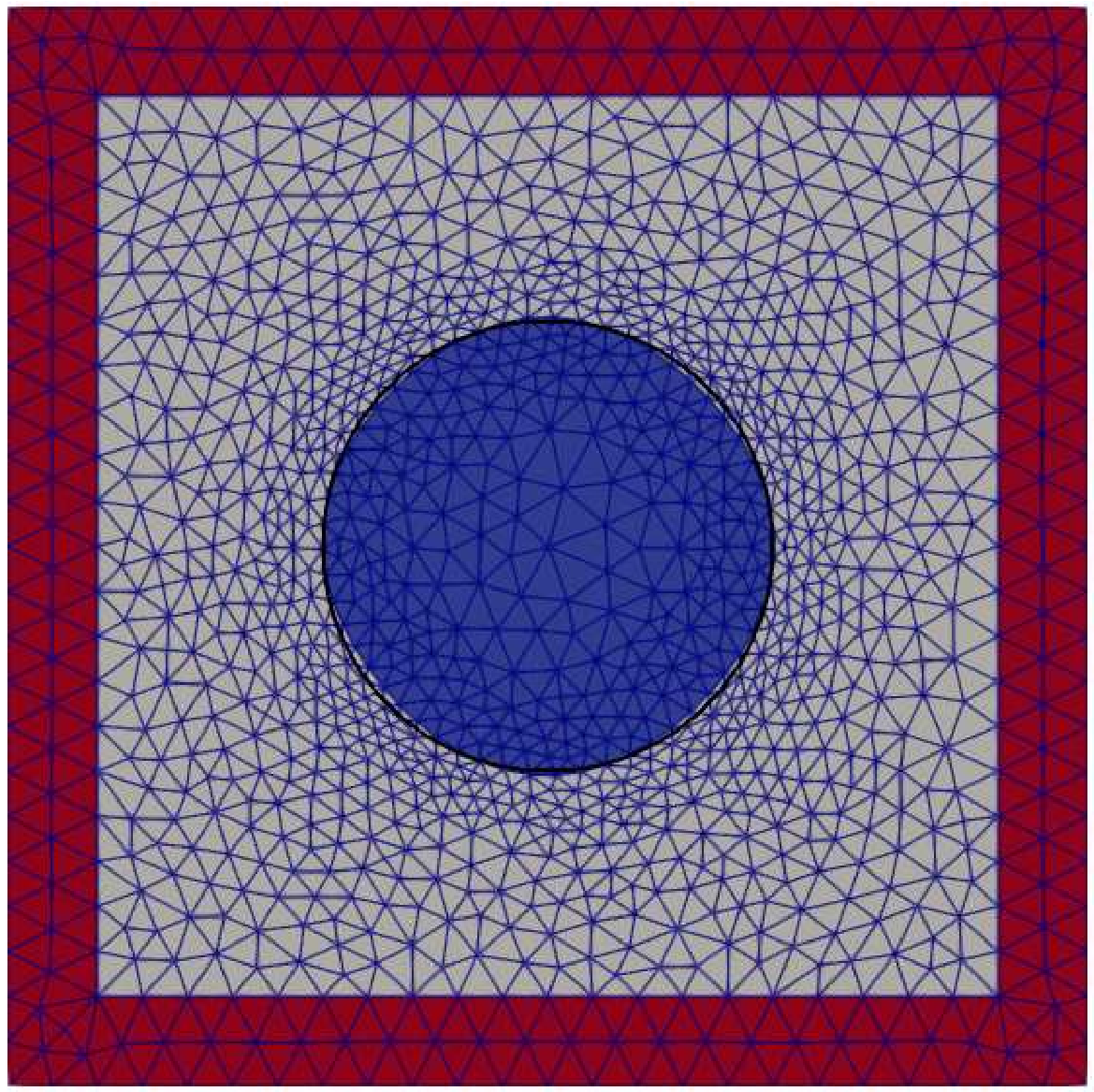}
				&\includegraphics[width = 0.2\textwidth]{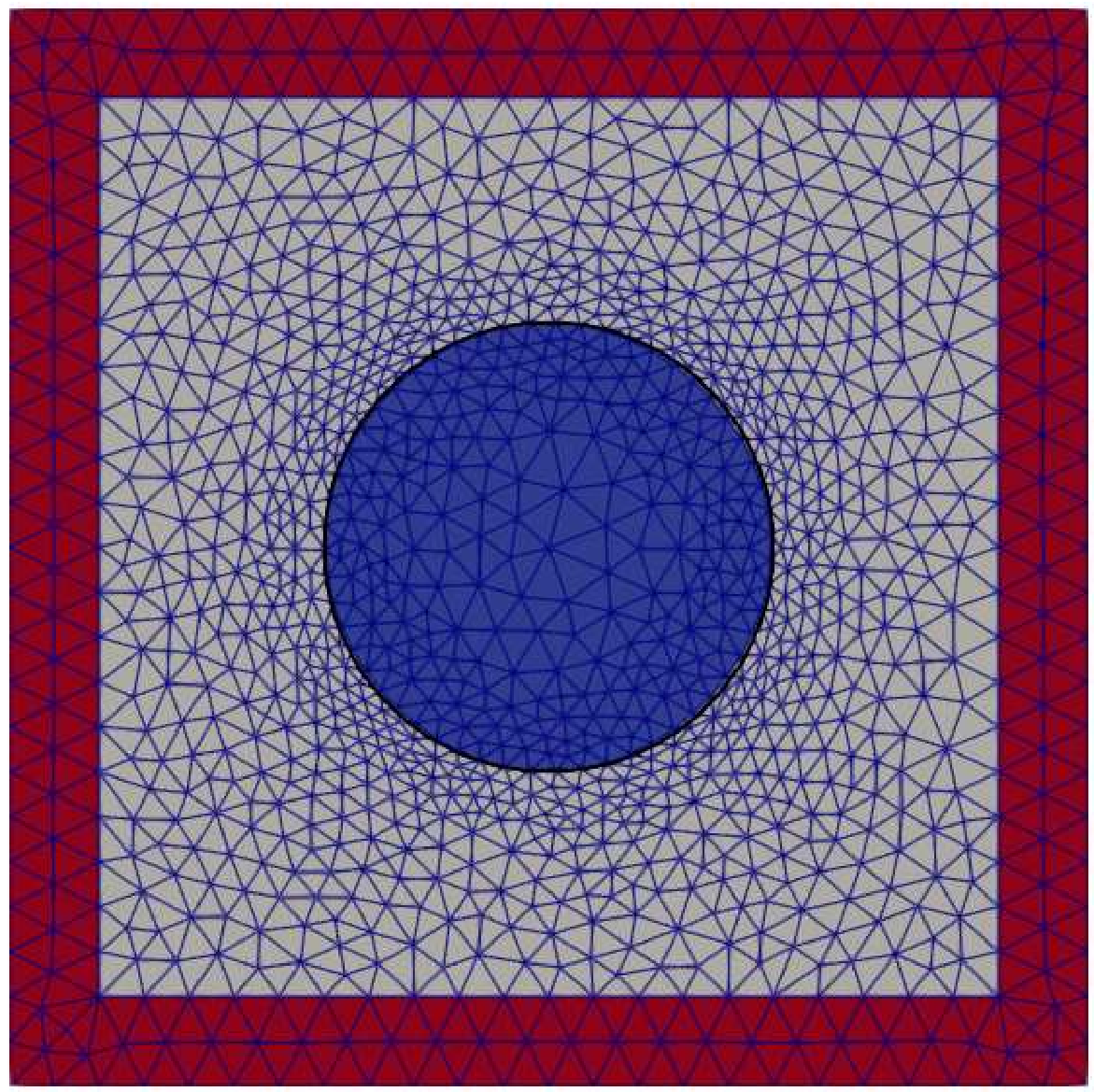}
				&\\
				& Iteration 10 & Iteration 22 & 
			\end{tabular}
		\end{small}
	\end{center}
	\caption{Example 2: Here, $\nlDom_1$ is the blue and $\nlDom_2$ the gray area. The nonlocal boundary $\nlBound$ is depicted in red. The black circle is the shape, that we try to reach. The given data $\data$ is the solution of the nonlocal Dirichlet problem, where the interface $\shape$ is the black circle and where we apply the singular symmetric kernel $\kernel_2$.}
	\label{fig:shape_ex_2}
\end{figure}\\~\\
The development of the objective functional values are illustrated in Figure \ref{fig:obj_fun_values} and the history of the deformation $\Wb^{k,h}_{\epsilon}$ in the $L^2$-norm is depicted in Picture \ref{fig:def_norms}. Here, we also tested the integrable kernel $\kernel_1$ with the same parameters as in the first example on the second set-up. Analogously, we also used the singular kernel $\kernel_2$ with the presented parameters on the starting domain of the first experiment.
\begin{figure}[h!] 
	\begin{center}
		\begin{small}
			\begin{tabular}{cc}
				\includegraphics[width = 0.50\textwidth]{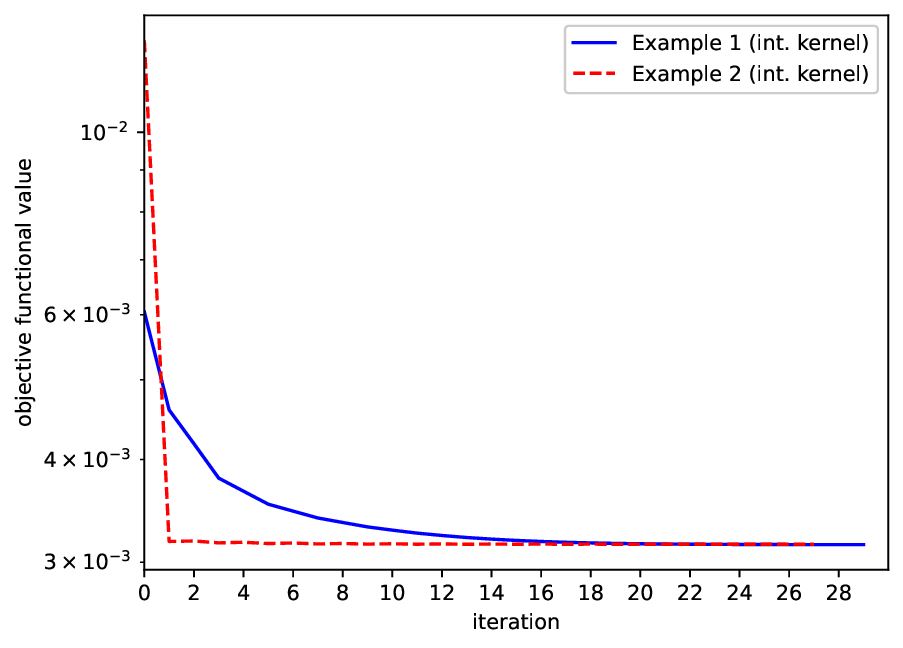}
				&\includegraphics[width = 0.50\textwidth]{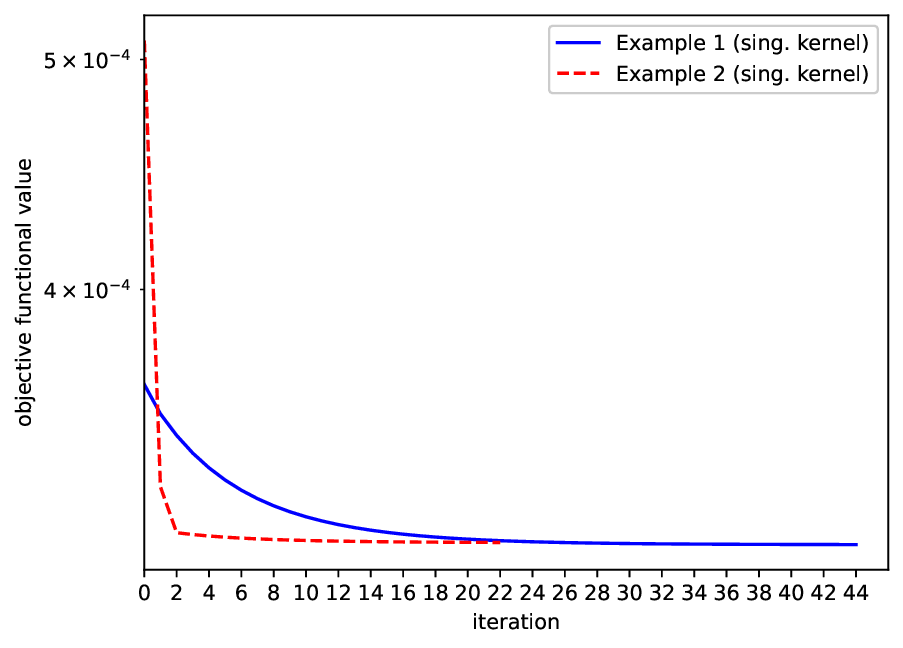}\\
				integrable kernel & singular kernel
			\end{tabular}
		\end{small}
	\end{center}
	\caption{Development of the objective functional values for both kernels starting on both initial shapes.}
	\label{fig:obj_fun_values}
\end{figure}
\begin{figure}[h!]
	\centering
	\def\svgwidth{0.4\textwidth}
	\includegraphics[width=0.6\textwidth]{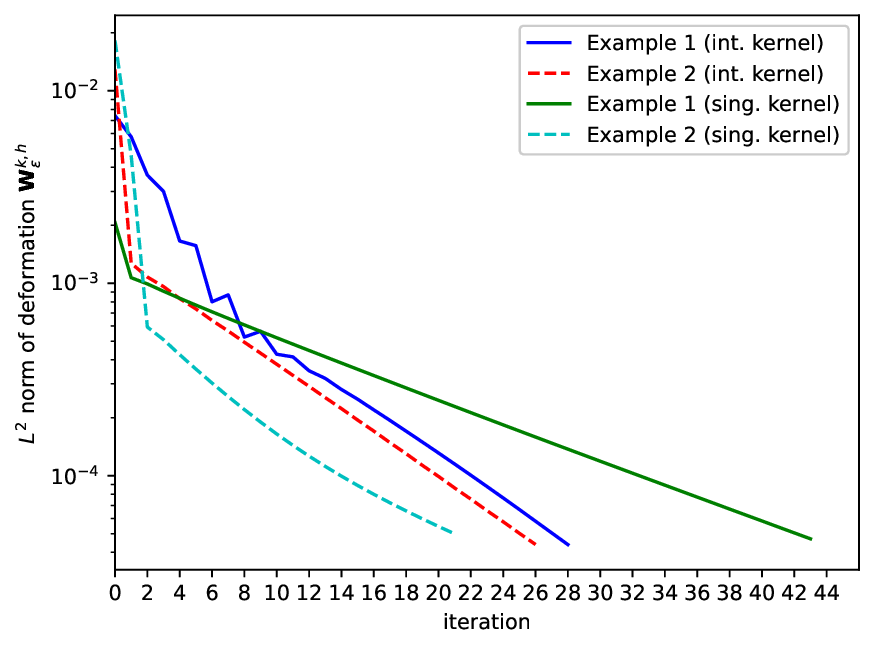}
	\caption{Development of $||\Wb^{k,h}_{\epsilon}||_2$ during the experiments.}
	\label{fig:def_norms}
\end{figure}
\newpage
\section{Conclusion}
 In this work, we investigated interface identification governed by a nonlocal Dirichlet problem. Further, we developed the second shape derivative of the reduced functional, that corresponds to this constrained interface identification problem, by applying a version of the averaged adjoint method. The second shape derivative was then used in a Newton-like second order shape optimization algorithm, which was successfully applied to two numerical examples in the previous chapter. Thus, to the best knowledge of the authors, this is the first study of the numerical usage of second shape derivatives for nonlocal problems.
\bibliographystyle{plain}
\bibliography{literature.bib}
\appendix
\section{Proofs}
Before we prove Lemma \ref{lemma:frechet_diff_bounded_domain}, we need the next statement, where we utilize spaces of $k$-times continuously differentiable vector fields with compact support, which we denote by
\begin{align*}
	{C_c^k(\holdAll,\Rd)\defas \{\Vb \in C(\holdAll,\Rd): \supp (\Vb) \subset \holdAll \text{ compact}\}} \text{ and } C_c^{\infty}(\holdAll,\Rd) \defas \bigcap_{k\in \Nbb_0} C_c^k(\holdAll,\Rd).
\end{align*} 
\begin{lemma} \label{lemma:C_c_dense_C_0}
	Let $k \in \Nbb_0$ and $\holdAll \subset \Rd$ be open and bounded. Then, $C_c^k(\holdAll,\Rd)$ is dense in $C_0^k(\holdAll,\Rd)$ w.r.t. the supremum norm.
\end{lemma}
\begin{proof}
	This proof is oriented on the proof of \cite[Proposition 4.35]{folland}.
	Given $f \in C_0^k(\holdAll,\Rd)$. For $n \in \Nbb$ we set $K_n \defas \{\xb \in \holdAll: |f(\xb)| \geq \frac{1}{n} \}$. Then, the sets $K_n$ are compact and by making use of \cite[8.18 The $C^{\infty}$ Urysohn Lemma]{folland}, we get functions $g_n \in C_c^{\infty}(\holdAll,[0.1])$ with $g_n = 1$ on $K_n$. As a next step, we define $f_n:\holdAll\rightarrow\Rd$, $f_n\defas fg_n$. Consequently, we derive $f_n \in C_c^k(\holdAll,\Rd)$ for $n \in \Nbb$ and $\sup_{\xb \in \holdAll} |f_n-f| < \frac{1}{n}$, which proves that $f_n \rightarrow f$ uniformly for $n \rightarrow \infty$.
\end{proof}
\subsection{Proof of Lemma \ref{lemma:frechet_diff_bounded_domain}}
\label{app:proof_frechet_diff}
\begin{proof} 
	After \cite[Proposition 2.32]{sokolowski_Introduction}.
	\begin{sloppypar}
	Since $C^{\infty}(\holdAll)\cap W^{1,1}(\holdAll)$ is dense in $W^{1,1}(\holdAll)$ (see \cite[Chapter 5.3 Theorem 2]{evansPDE}) regarding ${||\cdot||_{W^{1,1}(\holdAll)}}$ and ${C_c^1(\nlDom,\Rd)}$ is dense in $C_0^1(\holdAll,\Rd)$ w.r.t. the supremum norm (see Lemma \ref{lemma:C_c_dense_C_0}), we only need to show the result for ${f \in C^{\infty}(\holdAll)\cap W^{1,1}(\holdAll)}$ and ${\Vb \in C_c^1(\holdAll,\Rd)}$.
	Now, given $f \in C^{\infty}(\holdAll)\cap W^{1,1}(\holdAll)$ and ${\Vb \in C_c^1(\holdAll,\Rd)}$, applying the mean value theorem yields
	\begin{align*}
		f(\FbV_{t+r}(\xb)) - f(\FbV_r(\xb)) = \int_{0}^{1} t\grad f(\FbV_r(\xb) + st\Vb(\xb))^{\top} \Vb(\xb) ~ds
	\end{align*}
\end{sloppypar}
	for $\xb \in \holdAll$. Consequently, we derive
	\begin{align}
		&\int_{\holdAll}\left|\frac{1}{t}\left(f(\FbV_{r+t}(\xb)) - f(\FbV_r(\xb))\right) - \grad f(\FbV_r(\xb))^{\top} \Vb(\xb)\right| ~d\xb \nonumber \\
		&\leq \int_{\holdAll}  \int_{0}^{1} \left| \left( \grad f(\FbV_r(\xb) + st\Vb(\xb)) - \grad f(\FbV_r(\xb))\right)^{\top} \Vb(\xb) \right|~ds ~d\xb \nonumber \\
		&\leq \int_{0}^{1} \int_{\holdAll} \left| \left( \grad f(\FbV_r(\xb) + st\Vb(\xb)) - \grad f(\FbV_r(\xb))\right)^{\top} \Vb(\xb) \right| ~d\xb ~ds, \label{proof_lemma_frechet}
	\end{align}
	where we changed the order of integration in the last step. Now, we just have to show, that the double integral \eqref{proof_lemma_frechet} vanishes, when $t \rightarrow 0$. 
	For that reason we would like to prove, that the inner integral converges to zero. 
	Since $\Vb \in C_c^1(\holdAll,\Rd)$  and $f \in C^{\infty}(\holdAll)$ are continuous and thus bounded on the compact set $\supp (\Vb) \subset \holdAll$, the function $\grad f(\FbtV(\xb))^{\top} \Vb(\xb)$ is bounded for $(t,\xb) \in [0,T] \times \supp (\Vb)$ and thus for $(t,\xb) \in[0,T]\times\holdAll$, respectively.  Further, $\grad f(\FbV_{r+t}(\xb))$ converges pointwise to $\grad f(\FbV_r(\xb))$ as a composition of continuous functions, such that we can conclude by employing the dominated convergence theorem that
	\begin{align*}
		\lim_{t \rightarrow 0} \int_{\holdAll} \left| \left( \grad f(\FbV_r(\xb) + st\Vb(\xb)) - \grad f(\FbV_r(\xb))\right)^{\top} \Vb(\xb) \right| ~d\xb = 0 \quad \text{for all } s \in [0,1].
	\end{align*}
	Finally, utilizing dominated convergence again yields that
	\begin{align*}
		\lim_{t \rightarrow 0} \int_{0}^{1} \int_{\holdAll} \left| \left( \grad f(\FbV_r(\xb) + st\Vb(\xb)) - \grad f(\FbV_r(\xb))\right)^{\top} \Vb(\xb) \right| ~d\xb ~ds = 0.
	\end{align*}
\end{proof}~\\
In the two remaining proofs, we use the following observation:\\
Given a sequence $\{t_k\}_{k \in \Nbb} \in [0,T]^{\Nbb}$ with $\lim_{k \rightarrow \infty} t_k = 0$ and a domain $\widehat{\nlDom} \subset \Rd$, where ${d \in \Nbb}$. Moreover, assume two families of functions $\{g^t\}_{t \in [0,T]}$ and $\{h^{t_k}\}_{k \in \Nbb}$ with $g^t,h^{t_k} \in L^2(\widehat{\nlDom})$. Additionally, we suppose $g^t \rightarrow g^0$ in $L^2(\widehat{\nlDom})$ for $t \searrow 0$ and $h^{t_k} \rightharpoonup h^0$ in $L^2(\widehat{\nlDom})$ for $k \rightarrow \infty$. Then, it is easy to see that
\begin{align}\label{eq:weak_strong_conv}
	\lim_{k \rightarrow \infty, t \searrow 0} \int_{\widehat{\nlDom}} h^{t_k}g^t ~d\xb = \int_{\widehat{\nlDom}} h^0g^0 ~d\xb. 
\end{align}
Of course, if $h^{t_k} \rightarrow h^0$ in $L^2(\widehat{\nlDom})$, then \eqref{eq:weak_strong_conv} also holds. Further, if $h^{t_k} \rightarrow h^0$ in $L^\infty(\widehat{\nlDom})$ for  $h^{t_k} \in \L^{\infty}(\widehat{\nlDom})$, then
\begin{align}\label{eq:strong_linfty_conv}
	||g^th^{t_k}||_{L^2(\widehat{\nlDom})} \rightarrow ||g^0h^{0}||_{L^2(\widehat{\nlDom})} \text{ for } t \searrow 0. 
\end{align}
\begin{sloppypar}
In the following, \eqref{eq:strong_linfty_conv} will mainly be used to show that products involving $\xt \in L^{\infty}(\nlDom)$ still converge in $L^2(\nlDom)$. The $L^{\infty}(\nlDom)$-convergence of $\xt$ to $\xi^0=1$ is a direct consequence of ${\Vb \in \vecfieldsspecific}$ and the Leibniz rule for determinants.
\end{sloppypar}
\subsection{Proof of Lemma \ref{lemma:weak_conv}}
\label{app:weak_conv_proof}
\begin{proof}
	The proof for the first two assertions $\weakSol^t \rightharpoonup \weakSol^0$ and $\advar^t \rightharpoonup \advar^0$, for $t \searrow 0$, can be found in \cite{shape_paper}. We follow the same approach in order to show the statement for $\varphi^t$ and $\psi^t$. Therefore, we first prove the existence of a constant $C>0$ such that $||\varphi^t||_{L^2(\completeDom)}<C$ and $||\psi^t||_{L^2(\completeDom)}<C$ holds. In case of an integrable kernel that additionally fulfills Assumption (S1), we have some $C_1, C_2 > 0$ with
	\begin{align*}
		&||\varphi^t||^2_{L^2(\completeDom)} \leq C_1 |\varOp(t,\varphi^t,\varphi^t)| = C_1 |\varForceDer(t,\varphi^t) - \varOpDer(t,\weakSol^0,\varphi^t)| \\
		&\leq C_2\left(||f^t||_{H^1(\nlDom)} + ||\weakSol^0||_{L^2(\nlDom)}\right)||\varphi^t||_{L^2(\completeDom)},
	\end{align*}
	\begin{sloppypar}
	where we apply in the first step the norm equivalence as described in Remark \ref{remark:norm_equiv} and in the second that $\varphi$ fulfills the averaged adjoint equation \eqref{eq:AAE2_nonlocal}. Additionally, there exists a $\widehat{C}>0$ with $|\varOpDer(t,\weakSol^0,\varphi)|\leq \widehat{C}||\weakSol^0||_{L^2(\completeDom)}||\varphi^t||_{L^2(\completeDom)}$ (see Lemma \ref{lemma:varOpDer_continuity}) and ${|\varForceDer(t,\varphi^t)| \leq \widehat{C} || f^t||_{H^1(\nlDom)}||\varphi^t||_{L^2(\completeDom)}}$ due to the boundedness of $\xt$ and $\Vbt \in \vecfieldsspecific$ and the Cauchy-Schwarz inequality for $L^2$-functions, which we utilized for last inequality. Combined with $f^t \rightarrow f^0$ in $H^1(\nlDom)$ due to Lemma \ref{lemma:l2_convergence}, we get $||f^t||_{H^1(\nlDom)} < \widetilde{C}$ for some constant $\widetilde{C}>0$ and therefore we can conclude that $\{||\varphi^t||_{L^2(\completeDom)}\}_{t \in [0,T]}$ is bounded.\\
\end{sloppypar}
	If $\kernel_{\shape}$ is a singular symmetric kernel, we similarly derive for some $C,C_1,C_2,C_3 > 0$ that
	\begin{align*}
		|\varphi^t|_{H^s(\completeDom)}^2 \leq C_1|\varOp(t,\varphi^t,\varphi)| \leq C_2 (||f^t||_{H^1(\nlDom)} + |\weakSol^0|_{H^s(\completeDom)})|\varphi^t|_{H^s(\completeDom)} \leq C_3 |\varphi^t|_{H^s(\completeDom)} 
	\end{align*}
	where we used \cite[Lemma 4.3]{DuAnalysis}, i.e., there exists a constant $\widetilde{C}>0$ with $||\advar||_{L^2(\completeDom)} \leq \widetilde{C}|\advar|_{H^s(\completeDom)}$ for all $\advar \in \testSpace(\completeDom)$. As a result, we also draw the conclusion $||\varphi^t||_{L^2(\completeDom)} \leq \widetilde{C}|\varphi^t|_{H^s(\completeDom)} \leq C$ for an appropriate constant $C > 0$.\\
	The boundedness of $\{||\psi^t||_{L^2(\completeDom)}\}_{t\in [0,T]}$ can be proven analogously for both cases.\\
	Consequently, for every sequence $(t_n)_{n \in \Nbb} \in [0,T]^{\Nbb}$ with $\lim_{n \rightarrow \infty} t_n = 0$, there exists a subsequence $\{t_{n_k}\}_{k \in \Nbb}$ and functions $q_1,q_2 \in L^2(\completeDom)$ with $\varphi^{t_{n_k}} \rightharpoonup q_1$ and $\psi^{t_{n_k}} \rightharpoonup q_2$ for $k \rightarrow \infty$.\\
	For the proof of $\varOp(t,\varphi^{t_{n_k}},\widetilde{\advar}) \rightarrow \varOp(0,q_1,\widetilde{\advar})$ and $\varOp(t,\widetilde{\weakSol},\psi^{t_{n_k}}) \rightarrow \varOp(0,\widetilde{\weakSol}, q_2)$ for $k \rightarrow \infty$ we again refer to \cite{shape_paper}. Here, we show the remaining convergences by making use of similar arguments.\\
	Since $((\grad f^t)^{\top}\Vbt + \di \Vbt f^t)\xt \rightarrow \grad f_{\shape}^{\top}\Vb + \di \Vb f_{\shape}$ in $L^2(\nlDom)$ due to Lemma \ref{lemma:l2_convergence} and \eqref{eq:strong_linfty_conv}, we derive
	\begin{align*}
		\varForceDer(t_{n_k},\widetilde{\advar})=&\int_{\nlDom} \left((\grad f^{t_{n_k}})^{\top} \Vb^{t_{n_k}} \xi^{t_{n_k}} + \di \Vb^{t_{n_k}} f^{t_{n_k}} \xi^{t_{n_k}} \right) \widetilde{\advar} ~d\xb \\
		&\rightarrow \int_{\nlDom} \left(\grad f_{\shape}^{\top} \Vb + \di \Vb f_{\shape} \right) \widetilde{\advar} ~d\xb = \varForceDer(0,\widetilde{\advar}) \text{ for } k \rightarrow \infty.
	\end{align*}
	We now prove $\lim_{k \rightarrow \infty}\varOpDer(t_{n_k},\weakSol^0,\widetilde{\advar})=\varOpDer(0,\weakSol^0,\widetilde{\advar})$ separately for each class of kernels. For ease of presentation, we denote every subsequence of a sequence $(t_n)_{n \in \Nbb}$ by $(t_{n_k})_{k \in \Nbb}$ in the remaining part of this proof. However, technically they could be different subsequences or a subsequence of a subsequence.\\
	\textbf{Integrable kernels:}\\
	First, we recall that $\kernelt$ and $\grad \kernelt$ are essentially bounded and from Lemma \ref{lemma:l2_convergence} and observation \eqref{eq:strong_linfty_conv} follows
	\begin{align*}
		\Psi_1^{t}(\xb,\yb)\xt(\xb)\xt(\yb) = \grad \kernelt(\xb,\yb)^{\top}\Vbt(\xb,\yb) \xt(\xb)\xt(\yb) \rightarrow \grad \kernel^0(\xb,\yb)^{\top}\Vb^0(\xb,\yb)= \Psi_1^0(\xb,\yb)
	\end{align*}
	in ${L^2((\completeDom)^2)}$ for $t \searrow 0$. Then, for every sequence $(t_n)_{n\in\Nbb} [0,T]^{\Nbb}$ with $\lim_{n \rightarrow \infty}t_n = 0$ there exists a subsequence $(t_{n_k})_{k \in \Nbb}$ such that $\Psi_1^{t_{n_k}}\rightarrow \Psi_1^0$ almost everywhere on $(\completeDom)^2$. As a consequence of the dominated convergence theorem, we get
	\begin{align*}
		\frac{1}{2}\iint_{(\completeDom)^2} (\widetilde{\advar}(\xb) - \widetilde{\advar}(\yb))(&\weakSol^{0}(\xb)\Psi_1^{t_{n_k}}(\xb,\yb)  - \weakSol^{0}(\yb)\Psi_1^{t_{n_k}}(\yb,\xb))\xi^{t_{n_k}}(\xb)\xi^{t_{n_k}}(\yb) ~d\yb d\xb \\
		\rightarrow \frac{1}{2}\iint_{(\completeDom)^2} (\widetilde{\advar}(\xb) - &\widetilde{\advar}(\yb))(\weakSol^0(\xb) \Psi_1^0(\xb,\yb) - \weakSol^0(\yb) \Psi_1^0(\yb,\xb)) ~d\yb d\xb.
	\end{align*}
	Moreover, Lemma \ref{lemma:l2_convergence} and \eqref{eq:strong_linfty_conv} yield 
	\begin{align*}
		&\Psi_2^t(\xb,\yb)\xt(\xb) \xt(\yb)=\kernelt(\xb,\yb)(\di \Vbt(\xb) + \di \Vbt(\yb)) \xt(\xb) \xt(\yb) \\
		&\rightarrow \kernel^0(\xb,\yb)(\di \Vb^0(\xb) + \di \Vb^0(\yb))=\Psi_2^0(\xb,\yb) \text{ in } L^2((\completeDom)^2).
	\end{align*}
	By making use of the same argumentation as before, we conclude
	\begin{align*}
		\frac{1}{2}\iint_{(\completeDom)^2} (\widetilde{\advar}(\xb) - \widetilde{\advar}(\yb))&(\weakSol^0(\xb)\Psi_2^{t_{n_k}}(\xb,\yb) - \weakSol^0(\yb)\Psi_2^{t_{n_k}}(\yb,\xb)) \xi^{t_{n_k}}(\xb) \xi^{t_{n_k}}(\yb) ~d\yb d\xb\\
		\rightarrow \frac{1}{2}\iint_{(\completeDom)^2} (\widetilde{\advar}(\xb) - &\widetilde{\advar}(\yb))(\weakSol^0(\xb)\Psi_2^0(\xb,\yb) - \weakSol^0(\yb)\Psi_2^0(\yb,\xb)) ~d\yb d\xb.
	\end{align*}
	Consequently, we get $q_1 = \varphi^0$ because due to
	\begin{align}\label{eq:proof_q_1_varphi_0}
		\varOp(0,q_1, \widetilde{\advar}) = \lim_{k \rightarrow \infty} \varOp(t_{n_k}, \varphi^{t_{n_k}},\widetilde{\advar}) = \lim_{k \rightarrow \infty} \varForceDer(t_{n_k}, \widetilde{\advar}) - \varOpDer(t_{n_k},\weakSol^{t_{n_k}},\widetilde{\advar}) = \varForceDer(0, \widetilde{\advar}) - \varOpDer(0,\weakSol^0,\widetilde{\advar}),
	\end{align}
	the function $q_1$ fulfills the averaged adjoint equation \eqref{eq:AAE2} for $t=0$, where the solution is unique.
	Thus, we conclude $\varphi^t \rightharpoonup \varphi^0$ for $t \searrow 0$.\\
	\textbf{Singular symmetric kernels:}\\
	On every set $D_n$ the kernel $\kernel^t$ as well as the corresponding gradient $\grad \kernel^t$ are essentially bounded due to Assumption (S1) as explained in Remark \ref{remark:sing_kernel_bounded_Dn}. Consequently, it holds that
	\begin{align*}
		&\Psi_1^t(\xb,\yb)\xt(\xb)\xt(\yb)=(\grad \kernelt(\xb,\yb))^{\top}\Vbt(\xb,\yb) \xt(\xb)\xt(\yb) \in L^{\infty}(D_n) \text{ and}\\
		&\Psi_2^t(\xb,\yb)\xt(\xb) \xt(\yb)=\kernelt(\xb,\yb)(\di \Vbt(\xb) + \di \Vbt(\yb)) \xt(\xb) \xt(\yb) \in L^{\infty}(D_n).	
	\end{align*}
	With the same argumentation as in the case of integrable kernels, we derive for every sequence $(t_n)_{n \in \Nbb}$ the existence of a subsequence $(t_{n_k})_{k \in \Nbb}$ with 
	\begin{align*}
		\Psi_i^{t_{n_k}}(\xb,\yb) \xi^{t_{n_k}}(\xb) \xi^{t_{n_k}}(\yb)
		\rightarrow \Psi_i^0(\xb,\yb) \text{ in } L^2(D_n) \text{ for } i=1,2. 
	\end{align*}
	Then, by utilizing the dominated convergence theorem in the second step of the following calculation, we get
	\begin{align*}
		&\lim_{k\rightarrow \infty} \frac{1}{2}\iint_{(\completeDom)^2} (\weakSol^0(\xb) - \weakSol^0(\yb)) (\widetilde{\advar}(\xb) - \widetilde{\advar}(\yb))(\Psi_1^{t_{n_k}}(\xb,\yb) + \Psi_2^{t_{n_k}}(\xb,\yb)) \xi^{t_{n_k}}(\xb)\xi^{t_{n_k}}(\yb) ~d\yb d\xb\\
		&= \lim_{k,n \rightarrow \infty} \frac{1}{2}\iint_{D_n} (\weakSol^0(\xb) - \weakSol^0(\yb)) (\widetilde{\advar}(\xb) - \widetilde{\advar}(\yb))(\Psi_1^{t_{n_k}}(\xb,\yb) + \Psi_2^{t_{n_k}}(\xb,\yb)) \xi^{t_{n_k}}(\xb)\xi^{t_{n_k}}(\yb) ~d\yb d\xb \\
		&= \lim_{n \rightarrow \infty} \frac{1}{2}\iint_{D_n} (\weakSol^0(\xb) - \weakSol^0(\yb)) (\widetilde{\advar}(\xb) - \widetilde{\advar}(\yb))(\Psi_1^0(\xb,\yb) + \Psi_2^0(\xb,\yb)) ~d\yb d\xb \\
		&= \frac{1}{2}\iint_{(\completeDom)^2} (\weakSol^0(\xb) - \weakSol^0(\yb)) (\widetilde{\advar}(\xb) - \widetilde{\advar}(\yb))(\Psi_1^0(\xb,\yb) + \Psi_2^0(\xb,\yb)) ~d\yb d\xb.
	\end{align*}
	Thus, equation \eqref{eq:proof_q_1_varphi_0} also holds for singular symmetric kernels, which results in $q_1=\varphi^0$.\\~\\
	For both kernel cases, it can be shown quite analogously, that $q_2 = \psi^0$ and therefore $\psi^t \rightharpoonup \psi^0$ for $t \searrow 0$. However, for the convergence $\lim_{t \searrow 0}\varOpDer(t,\widetilde{\weakSol}, \advar^t) = \varOpDer(0, \widetilde{\weakSol}, \advar^0)$ we additionally have to make use of $\advar^t \rightharpoonup \advar^0$ and observation \eqref{eq:weak_strong_conv}.
\end{proof}
\subsection{Proof of Lemma \ref{lemma:D3_holds}}
\label{app:D3_proof}
\begin{proof}
	We show the statement by proving the convergence of each term
	\begin{align*}
		\lim_{k \rightarrow \infty,
			t \searrow 0} \partial_t &\reallagrangian(t, \weakSol^0, \advar^0, \psi^{s_{n_k}}, \varphi^{s_{n_k}}) \\
		=& \lim_{k \rightarrow \infty,
			t \searrow 0} \partial_t \objFunDer(t, \weakSol^0) - \partial_t \varForceDer(t, \advar^0) + \partial_t \varOpDer(t, \weakSol^0, \advar^0) \\
		&+ \partial_t \varOp(t, \weakSol^0, \psi^{s_{n_k}}) - \partial_t \varForce(t,\psi^{s_{n_k}}) + \partial_t\varOp(t, \varphi^{s_{n_k}}, \advar^0) - \partial_t\secondVarForce(t,\weakSol^0,\varphi^{s_{n_k}}).
	\end{align*}
	\begin{sloppypar}
	For the convergences 
	\begin{align*}
		\lim_{k \rightarrow \infty,
			t \searrow 0} \partial_t \varOp(t, \weakSol^0, \psi^{s_{n_k}}) = \partial_t \varOp(0, \weakSol^0, \psi^{0})\text{ and } \lim_{k \rightarrow \infty,
			t \searrow 0} \partial_t \varForce(t,\psi^{s_{n_k}}) = \partial_t \varForce(0,\psi^{0})
	\end{align*}
	we again refer to \cite{shape_paper}. Then, $\lim_{k \rightarrow \infty,
		t \searrow 0} \partial_t \varOp(t, \varphi^{s_{n_k}}, \advar^0) = \partial_t\varOp(0, \varphi^{0}, \advar^0)$ can be shown in a similar way. Moreover, Lemma \ref{lemma:l2_convergence} and \eqref{eq:strong_linfty_conv} yields $\xt \grad \data^t \rightarrow \grad \data$ in $L^2(\nlDom, \Rd)$ and ${\data^{t}\left. \frac{d}{dr} \right|_{r=t} \xi^r \rightarrow \data \di \Wb}$ in $L^2(\nlDom)$ for $t \searrow 0$. Furthermore, due to the boundedness of $\Wb$, we have $\xt(\grad \data^t)^{\top}\Wb \rightarrow \grad \data^{\top}\Wb$ in $L^2(\nlDom)$. Thus, by \eqref{eq:weak_strong_conv} we conclude
	\begin{align*}
		\lim_{k \rightarrow \infty,
			t \searrow 0} \partial_t\secondVarForce(t,\weakSol^0,\varphi^{s_{n_k}}) &= \lim_{k \rightarrow \infty,
			t \searrow 0}\int_{\nlDom} (\grad \data^t)^{\top} \Wb \varphi^{s_{n_k}} \xt - (\weakSol^0 - \data^t)\varphi^{s_{n_k}} \left.\frac{d}{dr} \right|_{r=t} \xi^r ~d\xb \\
		&= \int_{\nlDom} \grad \data^{\top} \Wb \varphi^{0} - (\weakSol^0 - \data)\varphi^{0} \di \Wb ~d\xb = \partial_t \secondVarForce(t,\weakSol^0,\varphi^0). 
	\end{align*}
\end{sloppypar}
	Additionally, since $\Vb \in \vecfieldsspecific$ and $\data \in H^2(\nlDom)$, we have $(\grad \data^t)^{\top}\Vbt \rightarrow (\grad \data)^{\top}\Vb$ in $L^2(\nlDom)$ and $\hess(\data^t)\Vbt \rightarrow \hess(\data)\Vb$ in $L^2(\nlDom,\R^2)$ as a result of Lemma \ref{lemma:l2_convergence}. Combined with the  boundedness of $\Wb \in \vecfieldsspecific$ we get $(\Vbt)^{\top} \hess(\data^t)\Wb \rightarrow \Vb^{\top} \hess(\data)\Wb$ in $L^2(\nlDom)$. Consequently, by applying \eqref{eq:weak_strong_conv} we have
	\begin{align*}
		&\lim_{k \rightarrow \infty,
			t \searrow 0} \partial_t \objFunDer(t, \weakSol^0) = \lim_{k \rightarrow \infty,
			t \searrow 0} \int_{\nlDom}  (\grad\data^t)^{\top}\Wb (\grad\data^t)^{\top}\Vbt \xt ~d\xb \\
		&- \int_{\nlDom} (\weakSol^0 - \data^t)\left((\Vbt)^{\top} \hess(\data)^t\Wb \xt + (\grad \data^t)^{\top}D\Vb^t\Wb\xt + (\grad\data^t)^{\top}\Vbt \left. \frac{d}{dr} \right|_{r=t}(\xi^r) \right) ~d\xb \\
		&+ \int_{\nlDom}-(\weakSol^0 - \data^t)(\grad\data^t)^{\top}\Wb \di \Vb^t \xt + \frac{1}{2} \left(\weakSol^0 - \data^t\right)^2  \left( (\grad \di \Vbt)^{\top}\Wb\xt + \di \Vbt \left. \frac{d}{dr} \right|_{r=t} \xi^r \right) ~d\xb\\
		&=\int_{\nlDom} \grad\data^{\top}\Wb \grad\data^{\top}\Vb ~d\xb \\
		&- \int_{\nlDom} (\weakSol^0 - \data)\left(\Vb^{\top} \hess(\data)\Wb + \grad \data^{\top}D\Vb\Wb + \grad\data^{\top}\Vb \di \Wb \right) ~d\xb \\
		&+ \int_{\nlDom}-(\weakSol^0 - \data)\grad\data^{\top}\Wb \di \Vb + \frac{1}{2} \left(\weakSol^0 - \data\right)^2 \left( (\grad \di \Vb)^{\top}\Wb + \di \Vb \di \Wb \right) ~d\xb
	\end{align*}
	Further, $f_{\shape} \in H^2(\nlDom)$ and with the same arguments as before, we derive
	\begin{align*}
		\lim_{k \rightarrow \infty,
			t \searrow 0}&\partial_t \varForceDer(t,\advar^0) \\
		=& \lim_{k \rightarrow \infty,
			t \searrow 0} \int_{\nlDom} (\Vbt)^{\top}\hess(f^t)\Wb \advar^0 \xt + (\grad f^t)^{\top} D\Vb^t\Wb\advar^0\xt + (\grad f^t)^{\top}\Vb^t\advar^0 \left. \frac{d}{dr} \right|_{r=t} (\xi^r) ~d\xb\\
		&+\int_{\nlDom} (\grad f^t)^{\top} \Wb \advar^0 \di \Vb^t \xt + f^t \advar^0 \left( \grad(\di \Vbt)^{\top}\Wb \xt + \di \Vbt \left. \frac{d}{dr} \right|_{r=t} \left(\xi^r \right) \right) ~d\xb\\
		=&\int_{\nlDom} \Vb^{\top}\hess(f_{\shape})\Wb \advar^0 + \grad f_{\shape}^{\top} D\Vb\Wb\advar^0 + \grad f_{\shape}^{\top}\Vb\advar^0 \di \Wb + \grad f_{\shape}^{\top} \Wb \advar^0 \di \Vb ~d\xb\\
		&+\int_{\nlDom} f_{\shape} \advar^0 \left((\grad \di \Vb)^{\top}\Wb + \di \Vb \di \Wb \right) ~d\xb
	\end{align*}
	\textbf{Integrable kernels:}\\
	First, we denote that the term
	\begin{align*}
		\grad \Psi_1^t(\xb,\yb)^{\top}\Wb(\xb,\yb)\xt(\xb)\xt(\yb) = &\Vbt(\xb,\yb)^{\top}\hess(\kernelt)(\xb,\yb)\Wb(\xb,\yb)\xt(\xb)\xt(\yb) \\ &+\grad\kernelt(\xb,\yb)^{\top}D\Vbt(\xb,\yb)\Wb(\xb,\yb)\xt(\xb)\xt(\yb) 
	\end{align*}
	is essentially bounded on $(\completeDom)^2$. Moreover, observation \eqref{eq:strong_linfty_conv} and Lemma \ref{lemma:l2_convergence} yields
	\begin{align*}
		&\grad \Psi_1^t(\xb,\yb)\xt(\xb)\xt(\yb) \rightarrow \grad \Psi_1^0(\xb,\yb) \text{ in } L^2((\completeDom)^2,\Rd) \text{ and }\\ &\Psi_1^t(\xb,\yb)\xt(\xb)\xt(\yb)=\grad\kernelt(\xb,\yb)^{\top}\Vbt(\xb,\yb)\xt(\xb)\xt(\yb) \rightarrow \grad\kernel^0(\xb,\yb)^{\top}\Vb^0(\xb,\yb) =\Psi_1^0(\xb,\yb)
	\end{align*}
	in $L^2((\completeDom)^2,\R)$. Thus, for every sequence $\{t_n\}_{n \in \Nbb}$ there exists a subsequence $\{t_{n_k}\}_{k \in \Nbb}$ such that $\Psi_1^{t_{n_k}}(\xb,\yb)\xi^{t_{n_k}}(\xb)\xi^{t_{n_k}}(\yb) \rightarrow \Psi_1^0(\xb,\yb)$ and $\grad \Psi_1^{t_{n_k}}(\xb,\yb)\xi^{t_{n_k}}(\xb)\xi^{t_{n_k}}(\yb) \rightarrow \grad \Psi_1^0(\xb,\yb)$ pointwise almost everywhere. By making use of the dominated convergence theorem we derive
	\begin{align*}
		\iint_{(\completeDom)^2} (\advar^0(\xb) - \advar^0(\yb))(&\weakSol^0(\xb)\grad \Psi_1^{t_{n_k}}(\xb,\yb)^{\top}\Wb(\xb,\yb) \\
		&- \weakSol^0(\yb)\grad \Psi_1^{t_{n_k}}(\yb,\xb)^{\top}\Wb(\yb,\xb))\xi^{t_{n_k}}(\xb)\xi^{t_{n_k}}(\yb) ~d\yb d\xb \\
		\rightarrow \iint_{(\completeDom)^2} (\advar^0(\xb) - \advar^0(\yb))(&\weakSol^0(\xb)\grad \Psi_1^0(\xb,\yb)^{\top}\Wb(\xb,\yb) \\
		&- \weakSol^0(\yb)\grad \Psi_1^0(\yb,\xb)\Wb(\yb,\xb)) ~d\yb d\xb \text{ for } k \rightarrow \infty.
	\end{align*}
	Moreover, since $\xt$ is continuously differentiable, the derivative $\left. \frac{d}{dr} \right|_{r=t} \xi^r$ is also essentially bounded and $\di \Vbt \rightarrow \di \Vb$ pointwise as a composition of continuous functions. Combined with arguments from above the following integrands are essentially bounded by a constant independent of $t \in [0,T]$ and again for every sequence $\{t_n\}_{n \in \Nbb}$ there exists a subsequence $\{t_{n_k}\}_{k \in \Nbb}$ such that the integrands converge pointwise almost everywhere. Thus, dominated convergence yields 
	\begin{align*}
		\iint\limits_{(\completeDom)^2} (\advar^0(\xb) - \advar^0(\yb))(&\weakSol^0(\xb) \Psi_1^{t_{n_k}}(\xb,\yb)- \weakSol^0(\yb)\Psi_1^{t_{n_k}}(\yb,\xb))\left. \frac{d}{dr} \right|_{r=t_{n_k}} (\xi^r(\xb) \xi^r(\yb)) ~d\yb d\xb\\
		\rightarrow \iint\limits_{(\completeDom)^2} (\advar^0(\xb) - \advar^0(\yb))(&\weakSol^0(\xb) \Psi_1^0(\xb,\yb) -\weakSol^0(\yb)\Psi_1^0(\yb,\xb))(\di \Wb(\xb) + \di \Wb(\yb)) ~d\yb d\xb
	\end{align*}
	and we also conclude
	\begin{align*}
		\iint\limits_{(\completeDom)^2} (\advar^0(\xb) - \advar^0(\yb))(&\weakSol^0(\xb)\grad \Psi_2^{t_{n_k}}(\xb,\yb)^{\top}\Wb(\xb,\yb) \\
		&- \weakSol^0(\yb) \grad \Psi_2^{t_{n_k}}(\yb,\xb)^{\top}\Wb(\yb,\xb))
		\xi^{t_{n_k}}(\xb)\xi^{t_{n_k}}(\yb) ~d\yb d\xb\\
		\rightarrow \iint\limits_{(\completeDom)^2} (\advar^0(\xb) - \advar^0(\yb))(&\weakSol^0(\xb)\grad \Psi_2^0(\xb,\yb)^{\top}\Wb(\xb,\yb) - \weakSol^0(\yb) \grad \Psi_2^0(\yb,\xb)^{\top}\Wb(\yb,\xb)) ~d\yb d\xb.
	\end{align*}
	For the last term, we derive by applying the dominated convergence theorem
	\begin{align*}
		&\iint\limits_{(\completeDom)^2} (\advar^0(\xb) - \advar^0(\yb))(\weakSol^0(\xb)\Psi_2^{t_{n_k}}(\xb,\yb) - \weakSol^0(\yb)\Psi_2^{t_{n_k}}(\yb,\xb))\left. \frac{d}{dr} \right|_{r=t_{n_k}}(\xi^r(\xb)\xi^r(\yb)) ~d\yb d\xb\\
		&\rightarrow\iint\limits_{(\completeDom)^2} (\advar^0(\xb) - \advar^0(\yb))(\weakSol^0(\xb)\Psi_2^0(\xb,\yb) - \weakSol^0(\yb)\Psi_2^0(\yb,\xb))(\di \Wb(\xb) + \di \Wb(\yb)) ~d\yb d\xb
	\end{align*}
	As a result, we derive $\partial_t \varOpDer(t_{n_k},\weakSol^0,\advar^0) \rightarrow \partial_t \varOpDer(0,\weakSol^0,\advar^0)$ for $k \rightarrow \infty$ and 
	\begin{align*}
	\partial_t \varOpDer(t,\weakSol^0,\advar^0) \rightarrow \partial_t \varOpDer(0,\weakSol^0,\advar^0) \text{ for } t \searrow 0.
	\end{align*}
\newpage
	\textbf{Singular symmetric kernels:}\\
	On the sets $D_n$ the functions 
	\begin{align*}
		a_1(t,\xb,\yb) &\defas \grad \Psi_1^t(\xb,\yb)^{\top}\Wb(\xb,\yb)\xt(\xb)\xt(\yb) \\
		&= \left(\Vbt(\xb,\yb)^{\top}\hess(\kernelt)(\xb,\yb) + \grad \kernelt(\xb,\yb)^{\top}D\Vbt(\xb,\yb)\right)\Wb(\xb,\yb)\xt(\xb)\xt(\yb),\\ 
		a_2(t,\xb,\yb) &\defas \Psi_1^t(\xb,\yb)\left. \frac{d}{dr} \right|_{r=t} \left( \xt(\xb) \xt(\yb)\right)  = \grad \kernelt(\xb,\yb)^{\top}\Vbt(\xb,\yb) \left. \frac{d}{dr} \right|_{r=t} \left( \xt(\xb) \xt(\yb)\right) \\
		a_3(t,\xb,\yb) &\defas \grad \Psi_2^t(\xb,\yb)^{\top}\Wb(\xb,\yb)\xt(\xb)\xt(\yb) \\
		&= (\di \Vbt(\xb) + \di \Vbt(\yb))\grad \kernelt(\xb,\yb)^{\top} \Wb(\xb,\yb)\xt(\xb)\xt(\yb) \\
		&~~~+ \kernelt(\xb,\yb) (\grad \di \Vbt(\xb) + \grad \di \Vbt(\yb))^{\top} \Wb(\xb,\yb)\xt(\xb)\xt(\yb) \text{ and}\\
		a_4(t,\xb,\yb) &\defas \Psi_2^t(\xb,\yb)\left. \frac{d}{dr} \right|_{r=t} \left( \xi^r(\xb) \xi^r(\yb) \right)  = \kernelt(\xb,\yb)(\di \Vbt(\xb) + \di \Vbt(\yb)) \left. \frac{d}{dr} \right|_{r=t} \left( \xi^r(\xb) \xi^r(\yb) \right)
	\end{align*}
	are essentially bounded (see Remark \ref{remark:sing_kernel_bounded_Dn} and Assumption (S1)). With the same argumentation as in the case of integrable kernels, for every sequence $(t_n)_{n \in \Nbb}$ there exists a subsequence $(t_{n_k})_{k \in \Nbb}$ such that $a_i(t_{n_k},\xb,\yb) \rightarrow a_i(0,\xb,\yb)$ pointwise almost everywhere on $D_n$ for $i=1,...,4$. Thus, we derive
	\begin{align*}
		&\sum_{i=1}^4 \lim_{k \rightarrow \infty} \iint_{(\completeDom)^2} (\advar^0(\xb) - \advar^0(\yb))(\weakSol^0(\xb) - \weakSol^0(\yb)) a_i(t_{n_k},\xb,\yb) ~d\yb d\xb \\
		= &\sum_{i=1}^4 \lim_{k,n \rightarrow \infty} \iint_{D_n} (\advar^0(\xb) - \advar^0(\yb))(\weakSol^0(\xb) - \weakSol^0(\yb)) a_i(t_{n_k},\xb,\yb) ~d\yb d\xb \\
		= &\sum_{i=1}^4 \lim_{n \rightarrow \infty} \iint_{D_n} (\advar^0(\xb) - \advar^0(\yb))(\weakSol^0(\xb) - \weakSol^0(\yb)) a_i(0,\xb,\yb) ~d\yb d\xb \\
		= &\sum_{i=1}^4 \iint_{(\completeDom)^2} (\advar^0(\xb) - \advar^0(\yb))(\weakSol^0(\xb) - \weakSol^0(\yb)) a_i(0,\xb,\yb) ~d\yb d\xb.
	\end{align*}
	Consequently, we get $\partial_t \varOpDer(t_{n_k},\weakSol^0,\advar^0) \rightarrow \partial_t \varOpDer(0,\weakSol^0,\advar^0)$ for $k \rightarrow \infty$ and thus
	 \begin{align*}
	 	\partial_t \varOpDer(t,\weakSol^0,\advar^0) \rightarrow \partial_t \varOpDer(0,\weakSol^0,\advar^0) \text{ for } t \searrow 0.
	 \end{align*}
\end{proof}
\end{document}